\documentclass [11pt,a4paper]{article}
\usepackage[ansinew]{inputenc}
\usepackage{amssymb}
\usepackage{amsmath}
\usepackage{graphicx}
\usepackage{amsthm}
\usepackage{cite}
\newtheorem{theorem}{\textbf{Theorem}}[section]
\newtheorem{lemma}{\textbf{Lemma}}[section]
\newtheorem{proposition}{\textbf{Proposition}}[section]
\newtheorem{corollary}{\textbf{Corollary}}[section]
\newtheorem{remark}{\textbf{Remark}}[section]
\newtheorem{definition}{\textbf{Definition}}[section]

\allowdisplaybreaks[4]

\def\be{\begin{equation}}
\def\ee{\end{equation}}
\def\bea{\begin{eqnarray}}
\def\eea{\end{eqnarray}}
\def\bt{\begin{theorem}}
\def\et{\end{theorem}}
\def\bl{\begin{lemma}}
\def\el{\end{lemma}}
\def\br{\begin{remark}}
\def\er{\end{remark}}
\def\bp{\begin{proposition}}
\def\ep{\end{proposition}}
\def\bc{\begin{corollary}}
\def\ec{\end{corollary}}
\def\bd{\begin{definition}}
\def\ed{\end{definition}}

\def\R{\mathbb{R}}

 \def\non{\nonumber }

\def \no#1#2#3 {{\bf #1} (#3), #2.}
\def \eds#1#2#3 {#1, #2, #3.}
\begin{document}

\title{Global weak solution for a coupled compressible Navier-Stokes
and Q-tensor system}

\author{
{\sc Dehua Wang} \footnote{Department of Mathematics, University of
Pittsburgh, Pittsburgh, PA, 15260, Email:
\textit{dwang@math.pitt.edu}.},
 {\sc Xiang Xu} \footnote{Department
of Mathematical Sciences, Carnegie Mellon University, Pittsburgh,
PA, 15213, Email: \textit{xuxiang@andrew.cmu.edu}.} \ and {\sc Cheng
Yu} \footnote{Department of Mathematics, University of Pittsburgh,
Pittsburgh, PA, 15260, Email: \textit{chy39@pitt.edu.}} }
\date{}
\maketitle

\begin{abstract}

In this paper, we study a coupled compressible
Navier-Stokes/$Q$-tensor system modeling the nematic liquid crystal
flow in a three-dimensional bounded spatial domain. The existence and long time dynamics of globally defined
weak solutions for the coupled system are established, using weak
convergence methods, compactness and interpolation arguments.
% in the spirit of Feireisl \cite{FNP01} and Lions \cite{L98}. 
The symmetry and traceless properties of the $Q$-tensor play key roles in this
process.

\bigskip

\bigskip


\end{abstract}

\maketitle

\section{Introduction} In this paper we consider the following hydrodynamic
system modeling the compressible nematic liquid crystal flow in a
bounded domain, which is composed of a coupled Navier-Stokes and
Q-tensor equations (see \cite{DOY08, TDY03}):
\bea \rho_t+\nabla\cdot(\rho{u})&=&0, \label{density equ} \\
(\rho{u})_t+\nabla\cdot(\rho{u}\otimes{u})+\nabla(P(\rho))&=&\mathcal{L}u-\nabla\cdot\big(L\nabla{Q}\odot\nabla{Q}
-\mathcal{F}(Q)I_3\big) \non\\
&&+L\nabla\cdot(Q\mathcal{H}(Q)-\mathcal{H}(Q)Q),
 \label{navier-stokes} \\
Q_t+u\cdot\nabla{Q}-\Omega{Q}+Q\Omega&=&\Gamma\mathcal{H}(Q).
 \label{order parameter} \eea
The system \eqref{density equ}-\eqref{order parameter} is subject to
the following initial conditions: \be (\rho, \rho{u},
Q)|_{t=0}=(\rho_0(x), q_0(x), Q_0(x)),  \ \ x \in U, \label{IC} \ee
with
%%%%%%%%%%%%%%%%%%%%%%%%%%%%%%%%%%%%%%%%%%%%%%%%%%%%%%%%%%%%%%%%%%%%%%%%%%%%%
\be Q_0 \in H^1(U), \ \ \ Q_0 \in S_0^{(3)} \ \ \mbox{ a.e. in } U,
\label{assumption 2 on IC} \ee and the following boundary conditions
\be u(x, t)=0, \ \ Q(x, t)=Q_0(x), \ \mbox{for } (x, t) \in \partial
U \times (0, \infty). \label{BC} \ee
%%%%%%%%%%%%%%%%%%%%%%%%%%%%%%%%%%%%%%%%%%%%%%%%%%%%%%%%%%%%%%%%%%%%%%%%%%%%%
 The following compatibility
condition is also imposed \be  \rho_0 \in L^{\gamma}(U), \ \rho_0
\geq 0; \, q_0 \in L^1(U), \ q_0=0 \ \mbox{if } \rho_0=0; \
\frac{|q_0|^2}{\rho_0} \in L^1(U). \label{compatibility condition}
\ee
%%%%%%%%%%%%%%%%%%%%%%%%%%%%%%%%%%%%%%%%%%%%%%%%%%%%%%%%%%%%%%%%%%%%%%%%%%%%%
Here $U \subset \mathbb{R}^3$ is a smooth bounded domain, $\rho: U
\times [0, +\infty) \rightarrow \mathbb{R}^1$ is the density
function of the fluid, $u: U \times [0, +\infty) \rightarrow
\mathbb{R}^3$ represents the velocity field of the fluid,
$P=\rho^\gamma$ stands for the pressure function with the adiabatic
constant $\gamma > 1$, and $Q: U \times (0, +\infty) \rightarrow
S_0^{(3)}$ is the order parameter, with $S_0^{(3)} \subset
\mathbb{M}^{3 \times 3}$ representing the space of $Q$-tensors in
dimension $3$, i.e.
\[ S_0^{(3)}=\{Q \in \mathbb{M}^{3 \times 3}; Q_{ij}=Q_{ji}, \, tr(Q)=0, \, i,j=1,\cdots, 3    \}.   \]
Throughout our paper, $\mbox{div}$ stands for the divergence
operator in $\mathbb{R}^3$ and $\mathcal{L}$ stands for the Lam\'{e}
operator:
\[ \mathcal{L}u=\nu\Delta{u}+(\nu+\lambda)\nabla\mbox{div}{u},  \]
where $\nu$ and $\lambda$ are shear viscosity and bulk viscosity
coefficients of the fluid, respectively, which satisfy the following
physical assumptions: \be \nu > 0, \ \ 2\nu+3\lambda \geq 0.
\label{relation between nu and lambda} \ee
 The $(i, j)$-th entry of the tensor
$\nabla{Q}\odot\nabla{Q}$ is
$\displaystyle\sum_{k,l=1}^3\nabla_iQ_{kl}\nabla_jQ_{kl}$, and $I_3
\subset \mathbb{M}^{3 \times 3}$ stands for the $3 \times 3$
identity matrix. Furthermore, $\mathcal{F}(Q)$ represents the free
energy density of the director field \be \mathcal{F}(Q)=
\frac{L}{2}|\nabla{Q}|^2+\frac{a}{2}tr(Q^2)-\frac{b}{3}tr(Q^3)+\frac{c}{4}tr^2(Q^2),
 \label{def of free energy} \ee and we denote
 \be \mathcal{H}(Q)=L\Delta Q-a
Q+b\Big[Q^2-\frac{I_3}{3}tr(Q^2)\Big]-cQ tr(Q^2).  \ee Here
$\Omega=\frac{\nabla{u}-\nabla^{T}u}{2}$ is the skew-symmetric part
of the rate of strain tensor. $L>0$, $\Gamma>0$, $a\in\mathbb{R}$,  $b>0$ and $c>0$ are
material-dependent elastic constants (c.f. \cite{MZ10}).

The celebrated hydrodynamic theory for nematic liquid crystals,
namely the Ericksen-Leslie theory, was developed between 1958 and
1968. Afterwards Lin \cite{Lin2} and Lin-Liu \cite{LL95,LL2} added a
penalization term to the Oseen-Frank energy functional to relax the
nonlinear constraint of unit vector length, and made a serious of
important analytic work, such as existence of global weak solutions,
partial regularity, etc. The corresponding compressible liquid
crystal flow was studied in Wang-Yu \cite {WY}, and also see \cite{LQ}. On
the other hand, quite recently, for a simplified Ericksen-Leslie
system with the nonlinear constaint of unit vector length,
Lin-Lin-Wang \cite{LLW} proved the existence of global weak
solutions that are smooth away from at most finitely many singular
times in any bounded smooth domain of $\R^2$, and results on
uniqueness of weak solutions were given in \cite{LW10, XZ11}.
Moreover, for the corresponding compressible flow in one-dimensional case, the
existence of global regular and weak solutions to the compressible
 flow of liquid crystals was obtained in \cite {DLWW, DWW}. 
 The strong solutions in three-dimensional case was also discussed in \cite{HW2,HuWu1,HuWu2,HWW}.

Besides the Ericksen-Leslie theory, there are alternative theories
that attempt to describe the nematic liquid crystal, among which the
most comprehensive description is the $Q$-tensor theory proposed by
P. G. De Gennes  in \cite{PG}. Roughly speaking, a $Q$-tensor is a symmetric and
traceless matrix which can be interpreted from the physical point of
view as a suitably normalized second-order moment of the probability
distribution function describing the orientation of rod-like liquid
crystal molecules (see \cite{BM, BZ} for details). The static theory
of $Q$ tensor has been extensively studied in \cite{BM, BZ, M10,
MZ10}.  On the other hand, the mathematical analysis of the
corresponding hydrodynamic system was studied in Paicu-Zarnescu \cite{PZ11, PZ12}. 
More precisely, they establish the existence of global
weak solutions to the coupled system of incompressible Navier-Stokes
equations and $Q$-tensors in both two and three dimensional cases,
as well as the existence of global regular solutions in two-dimensions.

 In this paper, we are interested in the compressible
version of the model studied in \cite{PZ12}. In the current case,
the fluid flow is governed by the compressible Navier-Stokes
equations, and the motion of the order-parameter $Q$ is described by
a parabolic type equation. It combines a usual equation describing
the flow of compressible fluid with extra nonlinear coupling terms.
These extra terms are induced elastic stresses from the elastic
energy through the transport, which is represented by the equation
of motion for the tensor order parameter $Q$:
\begin{equation*}
(\partial_t+u\cdot\nabla)Q-S(\nabla u,Q)=\Gamma \mathcal{H},
\label{parabolic equation}
\end{equation*}
where $\Gamma>0$ is a collective rotational diffusion constant. The
first term on the left hand side of the above equation is the
material derivative of $Q$, which is generalized by a second term
\begin{equation*}
S(\nabla
u,Q)=(\xi{A}+\Omega)\big(Q+\frac{I_3}{3}\big)+\big(Q+\frac{I_3}{3}\big)(\xi{A}-\Omega)-2\xi\big(Q+\frac{I_3}{3}\big)tr(Q\nabla{u}).
\end{equation*}
Here $A=\frac{\nabla{u}+\nabla^T{u}}{2}$ is the rate of strain
tensor. The term $S(\nabla u,Q)$ appears in the equation because the
order parameter distribution can be both rotated and stretched by
the flow gradients. $\xi$ is a constant which depends on the
molecular details of a given liquid crystal, which also measures the
ratio between the tumbling and aligning effect that a shear flow
would exert over the liquid crystal directors. The right hand side
of the equation \eqref{parabolic equation} describes the internal
relaxation of the order parameter towards the minimum of the free
energy. Furthermore, it is noted that in the uniaxial nematic phase,
when the magnitude of the order parameter $Q$ remains constant, the
coupled hydrodynamic system is reduced to the Ericksen-Leslie system
with the validity of Parodi's relation (see \cite{DOY08}). For the
sake of simplicity in mathematical analysis, we take $\xi=0$ in our
system. And we want to point out that the case for $\xi\neq 0$ is
mathematically much more challenging. There are no existing results
for the coupled system by compressible Navier-Stokes and
$Q$-tensors, and the goal of this paper is to establish the
existence of global weak solutions for the compressible coupled
system. We note that due to higher nonlinearities in the coupled
system \eqref{density equ}-\eqref{order parameter}, compared to
earlier works in \cite{WY, LQ}, it is more difficult to study the
current system mathematically.

Note that when $Q$ is absent in \eqref{density equ}-\eqref{order
parameter}, the system is reduced to the compressible Navier-Stokes
equations. For the multidimensional compressible Navier-Stokes
equations, early work by Matsumura and Nishida \cite{MN79,MN80,MN83}
established the global existence with the small initial data, and
later by Hoff \cite{H95JDE,H95,H97} for discontinuous initial data.
To remove the difficulties of large oscillations,
 Lions in \cite{L98} introduced the concept of  renormalized solutions  and proved the global existence of finite energy weak solutions
 for $\gamma>9/5$, where the vacuum is allowed initially, and then Feireisl, {\it et al},  in \cite{F04, FNP01, FP99} extended the existence
 results to $\gamma>3/2$.
% Brescha-Desjardins \cite{BD} intend to extend the existence results of Lions \cite{L98} to the case when pressure laws both depend on
%density and temperature. They proved the compactness of sequences of
%weak solutions for the full Navier-Stokes solutions modeling viscous
%compressible and heat conducting fluids. However, the existence of
%global weak solutions to the full Navier-Stokes equations is still
%open.
Since the compressible Navier-Stokes equations is a sub-system to
\eqref{density equ}-\eqref{order parameter}, one cannot expect
better result than those in \cite{F04, FNP01, FP99}. To this end, in
this paper we shall study the initial-boundary value problem for
large initial data in certain functional spaces with $\gamma>3/2$.
To achieve our goal, we will use a three-level approximation scheme
similar to that in \cite{F04, FNP01}, which consists of
Faedo-Galerkin approximation, artificial viscosity, and artificial
pressure (see also \cite{DT06, DT07, LQ, WY}). Then, following the
idea in \cite{FNP01}, we show that the uniform estimate of the
density $\rho^{\gamma+\alpha} $ in $L^{1}$ for some $\alpha>0$
ensures the vanishing of artificial pressure and the strong
compactness of the density.  We will establish the weak continuity
of the effective viscous flux for our systems similar to that for
compressible Navier-Stokes equations as in Lions and Feireisl in
\cite{F04, FNP01, L98} to remove the difficulty of possible large
oscillation of the density. To obtain the related lemma on effective
viscous flux, we have to make delicate analysis to deal with the
coupling and interaction between $Q-$tensor and the fluid velocity,
especially certain higher order terms arising from equation
\eqref{navier-stokes}. It is noted that we have to exploit the
structure of the system \eqref{density equ}-\eqref{order parameter},
and make use of certain special properties of $Q$-tensor, namely
symmetry and trace-free, to obtain the necessary a priori bounds for
$Q$ and the weak continuity for the effective viscous flux.

The remaining part of this paper is organized as follows. In Section
2, after the introduction of some preliminaries, we state the main
existence result of this paper, namely Theorem \ref{theorem on main
result}. In Sections 3-5, we study the three-level approximations,
namely Faedo-Galerkin, vanishing viscosity, and artificial pressure,
respectively. Finally, in section 6, we discuss briefly the long
time dynamics of the global weak solution.

%%%%%%%%%%%%%%%%%%%%%%%%%%%%%%%%%%%%%%%%%%%%%%%%%%%%%%%%%%%%%%%%%%%%%%%%%%%%%%%%%%%%%%%%%%%%%%%%%%%%%%%%%%%%
\section{Preliminaries}\setcounter{equation}{0}
Throughout this paper, we denote by $\langle\cdot ,  \cdot\rangle$
the scalar product between two vectors, and 
$$A:B=tr(A^{T}B)=tr(AB^{T})$$ represents the inner product between two
$3\times 3$ matrices $A$ and $B$, $\|\cdot\|_{L^2(U)}$ will be
shorthanded by $\|\cdot\|$ if necessary. We use the Frobenius norm
of a matrix $|Q| = \sqrt{tr(Q^2)}=\sqrt{Q_{ij}Q_{ij}}$ and Sobolev
spaces for $Q$-tensors are defined in terms of this norm. For
instance, $$L^2(U, S_0^{3}) = \{Q: U \rightarrow S_0^3,
\int_{U}|Q(x)|^2dx < \infty \}.$$  Meanwhile, we denote $\mathcal{D}$
as $C_0^\infty$, and $\mathcal{D}'$ in the sense of distributions.
We denote by $C$ and $C_i, i=0, 1, \cdots$ genetic constants which
may depend only on $U$, the coefficients of the system
\eqref{density equ}-\eqref{order parameter}, and the initial data
$(\rho_0, u_0, Q_0)$. Special dependence will be pointed out
explicitly in the text if necessary. Here and after, the Einstein
summation convention will be used. We also denote the total energy
by \be
\mathcal{E}(t)=\int_{U}\Big(\frac12\rho|u|^2(t)+\frac{\rho^{\gamma}(t)}{\gamma-1}\Big)
dx+\mathcal{G}(Q(t)), \ee where \be \mathcal{G}(Q(t))=\int_{U}
\left(\frac{L}{2}|\nabla{Q}|^2+\frac{a}{2}tr(Q^2)-\frac{b}{3}tr(Q^3)+\frac{c}{4}tr^2(Q^2)\right)dx.\ee
An important property of the coupling system \eqref{density
equ}--\eqref{BC} is that it has a basic energy law, which indicates
the dissipative nature of the system. It states that the total sum
of the kinetic and internal energy are dissipated due to viscosity
and internal elastic relaxation.
\begin{proposition} \label{proposition on basic energy law}
If $(\rho, u, Q)$ is a smooth solution of the problem \eqref{density
equ}-\eqref{BC}, then for any $ t > 0$, the following energy
dissipative law holds  \bea
\frac{d}{dt}\mathcal{E}(t)+\int_{U}\left(\nu|\nabla{u}|^2+(\nu+\lambda)|\mbox{div}\,{u}|^2\right)\,dx+\Gamma\int_{U}tr^2(\mathcal{H})dx=0.
\label{basic energy law} \eea
\end{proposition}
\begin{proof}
%First, notice that
%$\Delta{u}=\nabla\mbox{div}u-\nabla\times(\nabla\times{u})$, where
%$\nabla\times$ is the curl operator in $\mathbb{R}^3$.
Multiplying equation \eqref{navier-stokes} with $u$ then integrating
over $U$, using the density equation \eqref{density equ} and
boundary condition \eqref{BC} for $u$, we get after integration by
parts that \bea
\frac12\frac{d}{dt}\int_{U}\rho|u|^2dx&=&-(\nu+\lambda)\int_{U}|\mbox{div}u|^2dx
-\nu\int_{U}|\nabla{u}|^2dx+\int_{U}\rho^{\gamma}\mbox{div}{u}\,dx
\non\\
&&-L\int_{U} (u\cdot\nabla{Q}) : \Delta{Q} dx+\int_{U}\Big\langle u,
\nabla\big[\frac{a}{2}tr(Q^2)-\frac{b}{3}tr(Q^3)+\frac{c}{4}tr^2(Q^2)
\big] \Big\rangle dx \non\\
&&-L\int_{U} \nabla{u}: Q\Delta{Q} \, dx +L\int_{U} \nabla{u}:
\Delta{Q}Q \, dx.
 \label{part1 of basic energy law} \eea Next, we multiply equation
\eqref{order parameter} with -$\mathcal{H}$, then take the trace and
integrate over $U$. Since $\Omega+\Omega^{T}=0$, $Q^{T}=Q$,
$tr(Q)=0$, after integration by parts we have \bea
&&\frac{d}{dt}\mathcal{G}(Q(t))\non\\
&=&-\Gamma\int_{U}tr^2(\mathcal{H})dx+L\int_{U} (u\cdot\nabla{Q}):
\Delta{Q}\, dx-\int_{U}\Big\langle
u,\nabla\big[\frac{a}{2}tr(Q^2)-\frac{b}{3}tr(Q^3)+\frac{c}{4}tr^2(Q^2)
\big] \Big\rangle dx \non\\
&&-\frac{L}{2}\int_{U}(\nabla{u}Q+Q\nabla^{T}u):\Delta{Q}dx
+\frac{L}{2}\int_{U}(\nabla^{T}{u}Q+Q\nabla{u}):\Delta{Q}dx
\non\\
&=&-\Gamma\int_{U}tr^2(\mathcal{H})dx+L\int_{U} (u\cdot\nabla{Q}):
\Delta{Q}\, dx-\int_{U}\Big\langle
u,\nabla\big[\frac{a}{2}tr(Q^2)-\frac{b}{3}tr(Q^3)+\frac{c}{4}tr^2(Q^2)
\big] \Big\rangle dx \non\\
&&-L\int_{U}\nabla{u}: \Delta{Q}Q dx+L\int_{U}\nabla{u}: Q\Delta{Q}
dx.
 \label{part2 of basic energy law} \eea
Adding \eqref{part1 of basic energy law} and \eqref{part2 of basic
energy law} together, it yields \bea
&&\frac12\frac{d}{dt}\int_{U}\rho|u|^2dx+\frac{d}{dt}\mathcal{G}(Q(t))
\non\\
&=&-(\nu+\lambda)\int_{U}|\mbox{div}\,u|^2dx
-\nu\int_{U}|\nabla{u}|^2dx-\Gamma\int_{U}tr^2(\mathcal{H})dx
+\int_{U}\rho^{\gamma}\mbox{div}{u}\,dx.
 \label{part3 of basic energy law} \eea
%%%%%%%%%%%%%%%%%%%%%%%%%%%%%%%%%%%%%%%%%%%%%%%%%%%%%%%%%%%%%%%%%%%%%%%%%%%%%%%%%%%%%%%%%%%%%%%%%%%%%%%%%%%%%%%
Using the density equation again, it follows after integration by
parts several times that \bea
\int_{U}\rho^{\gamma}\mbox{div}{u}\,dx&=&
-\int_{U}\langle\gamma\rho^{\gamma-2}\nabla\rho, \rho{u}
\rangle\,dx=-\frac{\gamma}{\gamma-1}\int_{U}\langle
\nabla\rho^{\gamma-1}, \rho{u} \rangle\,dx
\non\\
&=&\frac{\gamma}{\gamma-1}\int_{U}\rho^{\gamma-1}
\mbox{div}(\rho{u}) \,dx
=-\frac{1}{\gamma-1}\frac{d}{dt}\int_{U}\rho^{\gamma}dx.
\label{part4 of basic energy law} \eea Consequently, we finish the
proof after combining \eqref{part3 of basic energy law} and
\eqref{part4 of basic energy law}.

\end{proof}

%\br \label{remark on norm of div and curl} Since $3\nu+2\lambda \geq
%0$, it follows that \be \int_{U}(2\nu+\lambda)|\mbox{div}u|^2
%+\nu|\nabla\times{u}|^2dx \geq \int_{U}\frac{\nu}{2}|\mbox{div}u|^2
%+\nu|\nabla\times{u}|^2dx \geq \frac{\nu}{2}\int_{U}|\nabla{u}|^2dx.
%\ee \er

It is worth pointing that the assumption $c>0$ is necessary from a
modeling point of view (see \cite{M10, MZ10}) so that the total
energy $\mathcal{E}$ is bounded from below.
%%%%%%%%%%%%%%%%%%%%%%%%%%%%%%%%%%%%%%%%%%%%%%%%%%%%%%%%%%%%%%%%%%%%%%%%%%%%%%%%%%%%%%%%%%%%%%%%
\bl\label{lemma on lower bound of total energy} For any smooth
solution $(\rho, u, Q)$ to the problem \eqref{density
equ}-\eqref{BC}, it holds \be \mathcal{E}(t) \geq
\int_{U}\left(\frac{\rho|u|^2}{2}+\frac{\rho^\gamma}{\gamma-1}\right)
dx+\frac{L}{2}\|\nabla{Q}(t)\|^2+\frac{c}{8}\int_{U}\Big[tr(Q^2)+\frac{2a}{c}-\frac{2b^2}{c^2}
\Big]^2dx-\frac{1}{2c^3}(b^2-ca)^2|U|, \ee where $|U|$ represents
the Lebesgue measure of the domain $U$. \el
%%%%%%%%%%%%%%%%%%%%%%%%%%%%%%%%%%%%%%%%%%%%%%%%%%%%%%%%%%%%%%%%%%%%%%%%%%%%%%%%%%%%%%%%%%%%%%%%
\begin{proof}
Since $Q \in S_0^3$, $Q$ has three real eigenvalues at each point :
$\lambda_1$, $\lambda_2$ and $-(\lambda_1+\lambda_2)$. Hence
$tr(Q^2)=2(\lambda_1^2+\lambda_2^2+\lambda_1\lambda_2)$,
$tr(Q^3)=-3\lambda_1\lambda_2(\lambda_1+\lambda_2)$. Notice that
\bea tr(Q^3)&=&-3\lambda_1\lambda_2(\lambda_1+\lambda_2) \leq
3(\lambda_1^2+\lambda_2^2+\lambda_1\lambda_2)\Big[\frac{\varepsilon(\lambda_1+\lambda_2)^2}{4}+\frac{1}{\varepsilon}
\Big] \non\\
&\leq&
3(\lambda_1^2+\lambda_2^2+\lambda_1\lambda_2)\Big[\frac{\varepsilon(\lambda_1^2+\lambda_2^2+\lambda_1\lambda_2)}{2}+\frac{1}{\varepsilon}
\Big]  \leq
\frac{3\varepsilon}{8}tr^2(Q^2)+\frac{3}{2\varepsilon}tr(Q^2).
\label{inequality on trace of Q} \eea Taking
$\varepsilon=\frac{c}{b}$ in \eqref{inequality on trace of Q}, then
we infer that \bea \mathcal{G}(Q)&\geq&
\frac{L}{2}\|\nabla{Q}\|^2+\int_{U}\frac{c}{8}tr^2(Q^2)-\Big(\frac{b^2}{2c}-\frac{a}{2}\Big)tr(Q^2)\
dx \non\\
&=&\frac{L}{2}\|\nabla{Q}\|^2+\frac{c}{8}\int_{U}\Big[tr(Q^2)+\frac{2a}{c}-\frac{2b^2}{c^2}
\Big]^2dx-\frac{1}{2c^3}(b^2-ca)^2|U|. \eea

\end{proof}

Consequently, using Proposition \ref{proposition on basic energy
law}, Lemma \ref{lemma on lower bound of total energy}, it is
straightforward to deduce the following a priori bounds for $Q$.

\begin{corollary}\label{remarkofQ}
For any smooth solution $(\rho, u, Q)$ to the problem \eqref{density
equ}-\eqref{BC}, it holds \be Q \in L^{10}(0, T; U)\cap L^\infty([0,
T]; H^1(U))\cap L^2([0, T]; H^2(U)), \ \ \nabla{Q} \in
L^{\frac{10}{3}}(0, T; U). \ee

\end{corollary}
%\br It infers from Proposition \ref{proposition on basic energy law}
%and Lemma \ref{lemma on lower bound of total energy} that
%$\mathcal{E}(t)$ serves as a Lyapunov functional to the problem
%\eqref{density equ}-\eqref{BC}.
%\er
\begin{proof}
First, using Proposition \ref{proposition on basic energy law} and
Lemma \ref{lemma on lower bound of total energy}, we have
\begin{eqnarray*}
&&\frac{L}{2}\|\nabla{Q}(t)\|^2+\frac{c}{8}\int_{U}\Big[tr(Q^2)+\frac{2a}{c}-\frac{2b^2}{c^2}
\Big]^2dx \\
 &\leq& \frac{1}{2c^3}(b^2-ca)^2|U|+\mathcal{E}(t) \leq \frac{1}{2c^3}(b^2-ca)^2|U|+\mathcal{E}(0), \end{eqnarray*}
hence $\nabla{Q} \in L^\infty(0, T; L^2(U))$. Meanwhile, using
Holder inequality, it is easy to get from the above inequality that
\begin{eqnarray*} \|Q(t)\|_{L^2(U)}^4 &\leq&
|U|\int_{U}tr^2(Q^2)\,dx \leq
2|U|\int_{U}\Big[tr(Q^2)+\frac{2a}{c}-\frac{2b^2}{c^2}
\Big]^2+\Big(\frac{2a}{c}-\frac{2b^2}{c^2} \Big)^2\,dx \\
&\leq& \frac{16|U|}{c}\Big[
\frac{1}{2c^3}(b^2-ca)^2|U|+\mathcal{E}(0)\Big]+\frac{8|U|^2}{c^4}(ac-b^2)^2,
\end{eqnarray*}
which indicates $Q \in L^\infty(0, T; L^2(U))$. Next, we observe
that
\begin{eqnarray*}
&&\frac{\Gamma}{2}\int_0^T\int_{U}L^2|\Delta{Q}(x, t)|^2\,dxdt \\
&\leq&
\Gamma\int_0^T\int_{U}tr^2(\mathcal{H})\,dxdt+\Gamma\int_0^T\int_{U}\Big|a
Q-b\Big[Q^2-\frac{I_3}{3}tr(Q^2)\Big]+cQ tr(Q^2) \Big|^2\,dxdt  \\
 &\leq& \mathcal{E}(0)-\mathcal{E}(t)+C\Gamma\int_0^T\|Q\|_{H^1(U)}^2\,dt \\
 &\leq&
\frac{1}{2c^3}(b^2-ca)^2|U|+\mathcal{E}(0)+CT.
\end{eqnarray*}
Here $C>0$ depends on $a, b, c, \Gamma, U$ and $\mathcal{E}(0)$.
Consequently, we know $\Delta{Q}\in L^2(0, T; L^2(U))$. Finally, we
infer from Gagliardo-Nirenberg inequality that \begin{eqnarray*}
&&\|Q\|_{L^{10}(U)} \leq
C\|Q\|_{L^6(U)}^{\frac45}\|\Delta{Q}\|_{L^2(U)}^{\frac15}+C\|Q\|_{L^6(U)}
\leq
C\|Q\|_{H^1(U)}^{\frac45}\|\Delta{Q}\|_{L^2(U)}^{\frac15}+C\|Q\|_{H^1(U)},
\\
&&\|\nabla{Q}\|_{L^{\frac{10}{3}}(U)}\leq
C\|\nabla{Q}\|_{L^2(U)}^{\frac25}\|\Delta{Q}\|_{L^2(U)}^{\frac35}+C\|\nabla{Q}\|_{L^2(U)},
\end{eqnarray*}
thus the proof is complete by noting that $Q \in L^\infty(0, T;
H^1(U))$ and $\Delta{Q}\in L^2(0, T; L^2(U))$.
\end{proof}

Next, we introduce the definition of finite energy weak solutions.
\begin{definition} \label{def of finite energy weak solution}
For any $T>0$, $(\rho, u, Q)$ is called a finite energy weak
solution to the problem \eqref{density equ}-\eqref{BC}, if the
following conditions are satisfied.
\begin{itemize}
\item $\rho \geq 0$, $\rho \in L^{\infty}([0, T];
L^{\gamma}(U))$, \;\;$u \in L^2([0, T]; H_0^1(U))$, \\$Q \in
L^\infty([0, T]; H^1(U))$ $\cap$ $L^2([0, T]; H^2(U))$
\\and $Q \in S_0^3$ a.e. in $ U\times[0, T]$.

%with $(\rho, \rho{u}, Q)(0, x)=(\rho_0(x), m_0(x), Q_0(x))$, for $x
%\in U$.

\item Equations \eqref{density equ}-\eqref{order parameter} are
valid in $\mathcal{D}'((0, T), U)$. Moreover, \eqref{density equ} is
valid in $\mathcal{D}'((0, T), \mathbb{R}^3)$ if $\rho, u$ are
extended to be zero on $\mathbb{R}^3\setminus U$;

\item The energy $\mathcal{E}$ is locally integrable on $(0, T)$ and the energy inequality
$$\frac{d}{dt}\mathcal{E}(t)+\int_{U}\left(\nu|\nabla{u}|^2+(\nu+\lambda)|\mbox{div}\,{u}|^2+\Gamma tr^2(\mathcal{H})\right)dx
\leq 0, \ \ \ \ \mbox{holds in} \ D'(0, T).    $$

\item For any function $g \in C^1(\mathbb{R}^+)$ with the property
\be \mbox{there exists a positive constant } M=M(g), \mbox{such that
} g'(z)=0, \text{ for all } z \geq M, \label{property of
renormailized} \ee the following renormalized form of the density
equation holds in $\mathcal{D}'((0, T), U)$ \be
g(\rho)_t+\mbox{div}(g(\rho)u)+(g'(\rho)\rho-g(\rho))\mbox{div}\,u=0.
\label{renormalized form of density equ} \ee

\end{itemize}
\end{definition}

Now we can state the main result of this paper on the existence of
global weak solutions.
\begin{theorem} \label{theorem on main result}
Suppose $\gamma > \frac{3}{2}$ and the compatibility condition
\eqref{compatibility condition} is satisfied. Then for any $T>0$,
the problem \eqref{density equ}-\eqref{BC} admits a finite energy
weak solution $(\rho, u, Q)$ on $(0, T) \times U$.
\end{theorem}

We shall prove Theorem \ref{theorem on main result} via a three-level 
approximation scheme  which consists of Faedo-Galerkin approximation, 
artificial viscosity, and artificial pressure, as well as the weak convergence method.

%%%%%%%%%%%%%%%%%%%%%%%%%%%%%%%%%%%%%%%%%%%%%%%%%%%%%%%%%%%%%%%%%%%%%%%%%%%%%%%%%%%%%%%%%%%%%%%%%%%%%%%%%%%%
\section{The Faedo-Galerkin Approximation}\setcounter{equation}{0}
\subsection{Approximate solutions}
In this section, our goal is to solve the following problem
%%%%%%%%%%%%%%%%%%%%%%%%%%%%%%%%%%%%%%%%%%%%%%%%%%%%%%%%%%%%%%%%%%%%%%%%%%%
\bea &&\rho_t+\mbox{div}(\rho u)=\varepsilon\Delta\rho,
  \label{density equ Galerkin} \\
&&(\rho{u})_t+\mbox{div}(\rho{u}\otimes{u})+\nabla{P}(\rho)+\delta\nabla\rho^{\beta}
+\varepsilon\nabla{\rho}\cdot\nabla{u}
 \non\\
 &&=\mathcal{L}u-\nabla\cdot\big(L\nabla{Q}\odot\nabla{Q}
-\mathcal{F}(Q)I_3\big)+L\nabla\cdot(Q\mathcal{H}(Q)-\mathcal{H}(Q)Q),
 \label{navier-stokes Galerkin} \\
&&Q_t+u\cdot\nabla{Q}-\Omega{Q}+Q\Omega=\Gamma\mathcal{H}(Q),
 \label{order parameter Galerkin}    \eea
with modified initial conditions: \be \rho|_{t=0}=\rho_0 \in
C^3(\bar{U}), \ 0 < \underline{\rho} \leq \rho_0(x) \leq \bar{\rho},
\ \frac{\partial\rho_0}{\partial{n}}\Big|_{\partial{U}}=0,
  \label{IC 1 Gakerkin} \ee
  \be \rho{u}|_{t=0}=q(x)\in C^2(\bar{U},
\mathbb{R}^3), \ Q|_{t=0}=Q_0(x), \ \  Q_0 \in H^1(U), \ \ Q_0 \in
S_0^3 \ \ \mbox{a.e. in } U.
 \label{IC 2 Galerkin} \ee
Here $\underline{\rho}$ and $\bar{\rho}$ are two positive constants.
And it is subject to the following boundary conditions \bea
&&\frac{\partial\rho}{\partial\vec{n}}\Big|_{\partial{U}}=0,
\label{BC 1 Galerkin} \\
&&u|_{\partial{U}}=0, \ Q|_{\partial{U}}=Q_0(x).
 \label{BC 2 Galerkin} \eea
%%%%%%%%%%%%%%%%%%%%%%%%%%%%%%%%%%%%%%%%%%%%%%%%%%%%%%%%%%%%%%%%%%%%%%%%%%%
 \br It is noted that (c.f. \cite{F04}) the extra term $\varepsilon\Delta\rho$
 appearing on the right-hand side of equation \eqref{density equ Galerkin} represents
 a ``vanishing viscosity" without any physical meaning. On the other
 hand, such mathematical operation converts the original hyperbolic
 equation \eqref{density equ} to a parabolic one such that one can
 expect better regularity results for $\rho$ at this point.
 Meanwhile, the extra quantity $\varepsilon\nabla\rho\cdot\nabla{u}$ in
equation \eqref{navier-stokes Galerkin} is added to cancel extra
terms to establish necessary energy laws (see \eqref{basic energy
law 2} below). The term $\delta\rho^\beta$ is added to achieve
higher integrability for $\rho$, which is shown in the next section.
\er

To begin with, using a standard argument shown in \cite{FNP01}, we
have the following existence result.
\begin{lemma} \label{lemma on Neumann pb for density}
For the initial-boundary value problem \eqref{density equ Galerkin},
\eqref{IC 1 Gakerkin} and \eqref{BC 1 Galerkin}, there exists a
mapping $\mathcal{S}=S(u):$ $C([0, T]; C^2(\bar{U}, \mathbb{R}^3))
\rightarrow C([0, T]; C^3(U))$ with the following properties:

\noindent(i) $\rho=\mathcal{S}(u)$ is the unique classical solution
of \eqref{density equ Galerkin}, \eqref{IC 1 Gakerkin} and \eqref{BC
1 Galerkin};

\noindent(ii)
$\underline{\rho}\exp\big(-\int_0^t\|\mbox{div}\,u(s)\|_{L^\infty(U)}ds\big)
\leq \rho(t, x) \leq
\bar{\rho}\exp\big(\int_0^t\|\mbox{div}\,u(s)\|_{L^\infty(U)}ds\big)$;

\noindent(iii) For any $u_1, u_2$ in the set
\[ \mathcal{M}_k=\{u \in C([0, T]; H_0^1(U)), \ \mbox{s.t. } \|u(t)\|_{L^\infty(U)}+\|\nabla{u}(t)\|_{L^\infty(U)}
 \leq k, \forall \, t  \}, \]
it holds \be\|\mathcal{S}(u_1)-\mathcal{S}(u_2)\|_{C([0, T];
H^1(U))} \leq Tc(k, T)\|u_1-u_2\|_{C([0, T]; H_0^1(U))}
\label{contraction map for density}\ee
\end{lemma}

Next, we shall provide the following lemma which is useful for
subsequent arguments in the Faedo-Galerkin approximate scheme.
%%%%%%%%%%%%%%%%%%%%%%%%%%%%%%%%%%%%%%%%%%%%%%%%%%%%%%%%%%%%%%%%%%%%%%%%%%%%%%%
\bl \label{lemma on solvability of Q} For each $u \in C([0, T];
C^2_0(\bar{U}, \mathbb{R}^3))$,  there exists a unique solution $Q$
$\in$ $L^\infty([0, T];$ $H^1(U))$ $\cap$ $L^2([0, T]; H^2(U))$ to
the initial boundary value problem \bea
&&Q_t+u\cdot\nabla{Q}-\Omega{Q}+Q\Omega=\Gamma\mathcal{H}(Q),
\label{order parameter 2} \\
&&Q|_{t=0}=Q_0(x), \ \ Q|_{\partial{U}}=Q_0,
 \eea
with $Q_0$ satisfies \eqref{assumption 2 on IC}. Moreover, the above
mapping $u \mapsto Q[u]$ is continuous from each bounded set of
$C([0, T];C^2_0(\bar{U}, \mathbb{R}^3))$ to $L^\infty([0, T];$
$H^1(U))$ $\cap$ $L^2([0, T]; H^2(U))$. Furthermore, $Q[u] \in
S_0^3$ a.e. in $U \times [0, T]$. \el

\begin{proof}For each $u \in C([0, T];
C^2_0(\bar{U}, \mathbb{R}^3))$, the existence of such $Q$ is
guaranteed by standard parabolic theory (c.f. \cite{LSU68}). To
prove $Q$ lies in $L^\infty([0, T];$ $H^1(U))$ $\cap$ $L^2([0, T];
H^2(U))$, suppose $\|u\|_{C(0, T; C^2_0(\bar{U}))} \leq M$ for some
positive constant M. We multiply equation \eqref{order parameter 2}
with $-\Delta{Q}$, then take the trace and integrate over $U$, using
Young's inequality, we get
%%%%%%%%%%%%%%%%%%%%%%%%%%%%%%%%%%%%%%%%%%%%%%%%%%%%%%%%%%%%%%%%%%%%%%%%%%%%%%%%
\bea \frac12\frac{d}{dt}\|\nabla{Q}\|^2+\Gamma{L}\|\Delta{Q}\|^2
&=&\int_{U}(u\cdot\nabla{Q}):\Delta{Q}\,dx+\int_{U}(Q\Omega):\Delta{Q}\,dx
-\int_{U}(\Omega{Q}):\Delta{Q}\,dx  \non\\
&&+\int_{U}\Big(a Q-bQ^2+\frac{b}{3}tr(Q^2){I_3}+cQ
tr(Q^2)\Big):\Delta{Q}\,dx \non\\
&\leq&\frac{\Gamma{L}}{4}\|\Delta{Q}\|^2+C_1\|Q\|_{H^1}^2. \non \eea
%%%%%%%%%%%%%%%%%%%%%%%%%%%%%%%%%%%%%%%%%%%%%%%%%%%%%%%%%%%%%%%%%%%%%%%%%%%%%%%%
Next, multiplying equation \eqref{order parameter 2} with $Q$, in a
similar way we have \bea \frac12\frac{d}{dt}\|Q\|^2
&=&\Gamma{L}\int_{U}\Delta{Q}:Q\,dx-\int_{U}(u\cdot\nabla{Q}):Q\,dx+\int_{U}\frac{a}{2}tr(Q^2)-\frac{b}{3}tr(Q^3)
+\frac{c}{4}tr^2(Q^2)\,dx \non\\
&\leq& \frac{\Gamma{L}}{4}\|\Delta{Q}\|^2+C_2\|Q\|_{H^1}^2. \non
\eea Here $C_1>0$ and $C_2>0$ are two constants which may depend on
$M$, $a$, $b$, $c$, $\Gamma$ and $L$. Summing up the above two
equations, we obtain \be
\frac{d}{dt}\|Q\|_{H^1}^2+\Gamma{L}\|\Delta{Q}\|^2 \leq
C\|Q\|_{H^1}^2. \non \ee Using Gronwall's inequality again, we infer
that \be \|Q\|_{L^\infty(0, T; H^1(U))}+\|Q\|_{L^2(0, T; H^2(U))}
\leq C^\ast, \ee where $C^\ast>0$ is a constant which may depend on
$M$, $\|Q_0\|_{H^1(U)}$, $a$, $b$, $c$, $\Gamma$, $L$ and $T$.

To prove uniqueness, suppose $Q_1$ and $Q_2$ are two different
solutions, then $\bar{Q}=Q_1-Q_2$ satisfies \bea
{\bar{Q}}_t+u\cdot\nabla\bar{Q}-\Omega\bar{Q}+\bar{Q}\Omega&=&\Gamma\Big(L\Delta
\bar{Q}-a\bar{Q}+b\big[Q_1^2-Q_2^2-\frac{I_3}{3}tr(Q_1^2-Q_2^2)\big]\non\\
&&\;\;\;\;-cQ_1 tr(Q_1^2)+cQ_2tr(Q_2^2)\Big),
\label{diff equ of order parameter} \\
\bar{Q}|_{t=0}&=&0, \ \ \bar{Q}|_{\partial{U}}=0.
 \eea
Multiplying both sides of equation \eqref{diff equ of order
parameter} with $\bar{Q}$, then taking its trace and integrating
over $U$, due to the assumption $u \in C([0, T]; C^2_0(\bar{U},
\mathbb{R}^3))$ and the fact that $ \|Q\|_{L^\infty(0, T;
H^1(U))}\leq C^\ast$, for $Q=Q_1, Q_2$, we get \bea
&&\frac12\frac{d}{dt}\|\bar{Q}\|^2+\Gamma{L}\|\nabla{\bar{Q}}\|^2\non\\
&=&-\int_{U}(u\cdot\nabla{\bar{Q}})
:\bar{Q}\,dx-\Gamma{a}\|\bar{Q}\|^2+\Gamma{b}\int_{U}[\bar{Q}(Q_1+Q_2)]:
\bar{Q}\,dx-\frac{\Gamma{b}}{3}\int_{U}[\bar{Q}(Q_1+Q_2)]tr(\bar{Q})\,dx  \non\\
&&-\Gamma{c}\int_{U}tr(\bar{Q}^2)tr(Q_1^2)+(Q_2:\bar{Q})\big(\bar{Q}:
(Q_1+Q_2) \big)\,dx \non\\
&\leq&M\|\nabla\bar{Q}\|\|\bar{Q}\|+\Gamma|a|\|\bar{Q}\|^2+\frac{4\Gamma{b}}{3}\|\bar{Q}\|_{L^6(U)}\|\bar{Q}\|\|Q_1+Q_2\|_{L^3(U)}\non\\
&&+2\Gamma{c}\|\bar{Q}\|_{L^6(U)}\|\bar{Q}\|\big(\|Q_1\|_{L^6(U)}^2+\|Q_2\|_{L^6(U)}^2 \big)       \non\\
&\leq&\frac{\Gamma{L}}{2}\|\nabla{\bar{Q}}\|^2+C\|\bar{Q}\|^2, \eea
where we used Sobolev embedding inequality, Poincar\'e inequality
and Young's inequality to obtain the last inequality. Here $C$ is a
positive constant which U, $M$, $a$, $b$, $c$, $\Gamma$ and $L$.
Hence we arrive at the uniqueness result by applying Gronwall's
inequality.

Then we let $\{u_n\}$ be a bounded sequence in $C^2_0(\bar{U},
\mathbb{R}^3)$, with $\|u_n\|_{C(0, T; C^2_0(\bar{U}))}$ $\leq M$,
$\forall \, n \in \mathbb{N}$, and \be \displaystyle\lim_{n
\rightarrow \infty}\|u_n-u\|_{C(0, T; C_0^2(\bar{U}))}=0,
  \label{limit of u_n} \ee for some $u \in C(0, T;
C^2_0(\bar{U}))$. For the mappings $u_n \mapsto Q_n$, $u \mapsto Q$,
we denote by $\bar{Q}_n=Q_n-Q$ and we are going to show that \be
\displaystyle\lim_{n\rightarrow \infty}\|\bar{Q}_n\|_{L^\infty(0, T;
H^1(U))}+\|\bar{Q}_n\|_{L^2(0, T; H^2(U))}=0. \label{continuity of
solution operator} \ee

Taking the difference of the equations given by $Q_n$ and $Q$, then
taking the inner product with $-\Delta{\bar{Q}_n}$. we have \bea
&&\frac12\frac{d}{dt}\|\nabla\bar{Q}_n\|^2+\Gamma{L}\|\Delta\bar{Q}_n\|^2
\non\\
&=&\int_{U}(u_n\cdot\nabla{Q_n}-u\cdot\nabla{Q})
:\Delta\bar{Q}_n\,dx-\int_{U}(Q_n\Omega_n-Q\Omega):\Delta{\bar{Q}_n}\,dx
\non\\
&&+\int_{U}(\Omega_n\bar{Q}_n-\Omega{Q}):\Delta\bar{Q}_n\,dx
+\Gamma{a}\int_{U}\bar{Q}_n: \Delta\bar{Q}_n\,dx
-\Gamma{b}\int_{U}[\bar{Q}_n(Q_n+Q)]:
\Delta\bar{Q}_n\,dx  \non\\
&&+\Gamma{c}\int_{U}\big(Q_ntr(Q_n^2)-Qtr(Q^2)\big): \Delta{\bar{Q}_n} \,dx \non\\
&\doteq&I_1+\cdots+I_6, \eea with
$\Omega_n=\frac{\nabla{u}_n-\nabla^{T}u_n}{2}$, $n=1, 2, \cdots$.
Notice that $\|Q_n\|_{L^\infty(0, T; H^1(U))}+\|Q_n\|_{L^2(0, T;
H^2(U))} \leq C^\ast$ uniformly for $n \in \mathbb{N}$, we can
estimate $I_1$ to $I_6$ as follows: \bea I_1
&\leq&\int_{U}\left(|u_n\cdot\nabla\bar{Q}_n||\Delta\bar{Q}_n|+|u_n-u||\nabla{Q}||\Delta\bar{Q}_n|\right)
\,dx \non\\
&\leq&\int_{U}\left(\|u_n\|_{L^\infty(U)}|\nabla\bar{Q}_n||\Delta{\bar{Q}}_n|
+\|u_n-u\|_{L^\infty(U)}|\nabla\bar{Q}||\Delta{\bar{Q}}_n|\right)
\,dx \non\\
&\leq&M\|\nabla\bar{Q}_n\|\|\Delta\bar{Q}_n\|+\|u_n-u\|_{L^\infty(U)}
\int_{U}\left(\frac{12M}{\Gamma{L}}|\nabla\bar{Q}|^2+\frac{\Gamma{L}}{48M}|\Delta\bar{Q}_n|^2\right)\,dx
\non\\
&\leq&\frac{\Gamma{L}}{12}\|\Delta\bar{Q}_n\|^2+\frac{12M(C^\ast)^2}{\Gamma{L}}\|u_n-u\|_{L^\infty(U)}
+C\|\nabla\bar{Q}_n\|^2. \non \eea
%%%%%%%%%%%%%%%%%%%%%%%%%%%%%%%%%%%%%%%%%%%%%%%%%%%%%%%%%%%%%%%%%%%%%%%%%%%%%%%%%%%%%%%%%%%%%%%%%%%
\bea I_2
&\leq&\int_{U}\left(|\bar{Q}_n||\Omega_n||\Delta\bar{Q}_n|+|Q||\Omega_n-\Omega||\Delta\bar{Q}_n|
\right)\,dx \non\\
&\leq&M\|\bar{Q}_n\|\|\Delta\bar{Q}_n\|+\|\nabla{u}_n-\nabla{u}\|_{L^\infty(U)}
\int_{U}\left(\frac{12M}{\Gamma{L}}|Q|^2+\frac{\Gamma{L}}{48M}|\Delta\bar{Q}_n|^2\right)\,dx
\non\\
&\leq&\frac{\Gamma{L}}{12}\|\Delta\bar{Q}_n\|^2+C\|\nabla{u}_n-\nabla{u}\|_{L^\infty(U)}
+C\|\nabla\bar{Q}_n\|^2, \non
 \eea
where we used Poincar\'{e} inequality in the last step since
$\bar{Q}_n|_{\partial U}=0$. In the same way as $I_2$, we get \be
I_3 \leq
\frac{\Gamma{L}}{12}\|\Delta\bar{Q}_n\|^2+C\|\nabla{u}_n-\nabla{u}\|_{L^\infty(U)}
+C\|\nabla\bar{Q}_n\|^2. \non \ee For $I_4$ and $I_5$, using
Poincar\'{e} inequality again, it yields \be I_4 \leq
\frac{\Gamma{L}}{12}\|\Delta\bar{Q}_n\|^2+C\|\nabla\bar{Q}_n\|^2,
\non \ee  \bea I_5 \leq
\Gamma{b}\|Q_n+Q\|_{L^6(U)}\|\bar{Q_n}\|_{L^3(U)}\|\Delta
\bar{Q}_n\| \leq
\frac{\Gamma{L}}{12}\|\Delta\bar{Q}_n\|^2+C\|\nabla\bar{Q}_n\|^2.
\non \eea And \bea I_6
 &\leq&\Gamma{c}\|Q_n\|_{L^6(U)}^2\|\bar{Q_n}\|_{L^6(U)}\|\Delta \bar{Q_n}\|+
 \Gamma{c}\|Q\|_{L^6(U)}\|Q_n+Q\|_{L^6(U)}\|\bar{Q_n}\|_{L^6(U)}\|\Delta\bar{Q_n}\| \non\\
 &\leq&
\frac{\Gamma{L}}{12}\|\Delta\bar{Q}_n\|^2+C\|\nabla\bar{Q}_n\|^2.
\non \eea Putting all these estimates together, we get \be
\frac{d}{dt}\|\nabla\bar{Q}_n\|^2+\Gamma{L}\|\Delta\bar{Q}_n\|^2
\leq C\|u_n-u\|_{L^\infty(0, T; C^2_0(U))}+C\|\nabla\bar{Q}_n\|^2.
\non \ee Therefore, we conclude from Gronwall's inequality that \be
\|\nabla\bar{Q}_n\|^2(t)+\int_0^{T}\|\Delta\bar{Q}_n\|^2dt \leq
e^{CT}\|u_n-u\|_{L^\infty(0, T; C^2_0(U))}, \ \ \ \forall \, t \in
[0, T]. \ee Hence we can prove \eqref{continuity of solution
operator} by passing $n \rightarrow \infty$.

To finish the proof of this lemma, we finally show that $Q \in
S_0^3$, namely, $Q=Q^T$ and $tr(Q)=0$ a.e. in $U \times [0, T]$. It
is easy to observe that if $Q$ is a solution to \eqref{order
parameter 2}, so is $Q^T$. Hence $Q=Q^T$ a.e. by the aforementioned
uniqueness result. Then taking trace to both sides of the equation
\eqref{order parameter 2}, using the property $\Omega=-\Omega^{T}$
and $Q=Q^T$, we have \bea \frac{\partial}{\partial
t}tr(Q)-u\cdot\nabla{tr(Q)}&=&\Gamma{L}\Delta{tr(Q)}-\Gamma{a}\,tr(Q)-\Gamma{c}\,tr(Q)tr(Q^2),
\non\\
tr(Q)|_{t=0}&=&0, \ \ \ tr(Q)|_{\partial{U}}=0. \non \eea
Consequently, after multiplying both sides of the above equation
with $tr(Q)$ and integration over $U$, we can complete the proof by
the initial and boundary conditions and Gronwall's inequality.

\end{proof}

%%%%%%%%%%%%%%%%%%%%%%%%%%%%%%%%%%%%%%%%%%%%%%%%%%%%%%%%%%%%%%%%%%%%%%%%%%%
We proceed to solve \eqref{density equ Galerkin}-\eqref{BC 2
Galerkin} by the Faedo-Gelerkin approximation scheme. Let $\{\psi_n
\}_{n=1}^\infty \subset C^\infty(U, \mathbb{R}^3)$ be the
eigenfunctions of the Laplacian operator that vanish on the
boundary:
\[ -\Delta\psi_n=\lambda_n\psi_n \ \ \mbox{in } U, \ \ \psi_n|_{\partial{U}}=0. \]
Here $0<\lambda_1\leq\lambda_2\leq ...$ are eigenvalues and
$\{\psi_n\}_{n=1}^\infty$ forms an orthogonal basis of $H_0^1(U)$.
Let $X_n \doteq span\{\psi_1, \cdots, \psi_n\}$, $n=1, 2, \ldots$ be
a sequence of finite dimensional spaces.

Then we consider the following variational approximate problem for
$u_n \in C([0, T], X_n)$: $\forall \, t \in [0, T], \forall \, \psi
\in X_n$, \bea &&\int_{U}\langle \rho{u}_n(t), \psi \rangle
\,dx-\int_{U}\langle q, \psi \rangle
\,dx   \non\\
&=&\int_0^t\int_{U}\big\langle
\mathcal{L}{u}_n-\mbox{div}(\rho{u}_n\otimes{u}_n)-(\rho^\gamma+\delta\rho^\beta)-\varepsilon\nabla\rho\cdot\nabla{u}_n,
\psi \big\rangle\,dxds \non\\
&&-\int_0^t\int_{U}\left\langle\nabla\cdot\big(L\nabla{Q_n}\odot\nabla{Q_n}
-\mathcal{F}(Q_n)I_3\big), \psi\right\rangle dxds
\non\\
&&-L\int_0^t\int_{U}\left\langle\nabla\cdot(Q_n\mathcal{H}(Q_n)-\mathcal{H}(Q_n)Q_n),
\psi\right\rangle dxds.  \label{var prob for u} \eea Next, following
the idea in \cite{FNP01}, we introduce a family of operators
\[\mathcal{M}[\rho]: X_n \mapsto X_n^\ast, \ \mathcal{M}[\rho]v(w)=
\int_{U}\langle \rho{v}, w \rangle dx, \ \forall \, v, w \in X_n. \]
Here the existence and uniqueness of the solution $Q_n$ to
\eqref{order parameter Galerkin} is guaranteed by Lemma \ref{lemma
on solvability of Q}, while $\rho=\mathcal{S}(u_n)$ is the unique
classical solution to \eqref{density equ Galerkin} given by Lemma
\ref{lemma on Neumann pb for density}.

And it follows from the arguments in \cite{FNP01} that the map
\[ \rho \mapsto \mathcal{M}^{-1}[\rho] \]
from $N_{\eta}=\{\rho \in L^1(U)\, | \,\displaystyle\inf_{x\in
U}\rho \geq \eta>0\}$ is well defined and satisfies
%%%%%%%%%%%%%%%%%%%%%%%%%%%%%%%%%%%%%%%%%%%%%%%%%%%%%%%%%%%%%%%%%%%%%%%%%%%%%%%%%%%%%%%%%%%%%
\be
\big\|\mathcal{M}^{-1}[\rho^1]-\mathcal{M}^{-1}[\rho^2]\big\|_{\mathcal{L}(X_n^\ast,
X_n)} \leq C(n, \eta)\|\rho^1-\rho^2\|_{L^1(U)}.
 \label{contraction map property} \ee
Meanwhile, due to Lemma \ref{lemma on Neumann pb for density}, we
may rewrite the variational problem \eqref{var prob for u} as:
$\forall \, t \in [0, T], \forall \, \psi \in X_n$, \bea
 u_n(t)=\mathcal{M}^{-1}[\mathcal{S}(u_n)(t)]\Big(q^\ast
 +\int_0^t\mathcal{N}[\mathcal{S}(\rho_n(s), u_n(s), Q_n(s)]\,ds
 \Big), \label{vari prob 2 for u}
\eea with \bea \langle \mathcal{N}[\rho_n, u_n, Q_n], \psi \rangle
&=& \int_{U}\big\langle
\mathcal{L}{u}_n-\mbox{div}(\rho_n{u}_n\otimes{u}_n)-(\rho_n^\gamma+\delta\rho_n^\beta)-\varepsilon\nabla\rho_n\cdot\nabla{u}_n,
\psi \big\rangle\,dx \non\\
&&-\int_{U}\left\langle\nabla\cdot\big(L\nabla{Q_n}\odot\nabla{Q_n}
-\mathcal{F}(Q_n)I_3\big), \psi\right\rangle dx
\non\\
&&-L\int_{U}\left\langle\nabla\cdot(Q_n\mathcal{H}(Q_n)-\mathcal{H}(Q_n)Q_n),
\psi\right\rangle dx, \non\\
\rho_n&=&\mathcal{S}(u_n), \ \ Q_n=Q_n[S_n], \ \ q^\ast \in
X_n^\ast, \ \ \mbox{and } \ q^\ast(\psi)=\int_{U}\langle q, \psi
\rangle\,dx. \non
 \eea
Therefore, in view of \eqref{contraction map for density} and
\eqref{contraction map property}, using standard fixed point theorem
on $C([0, T], X_n)$, we obtain a local solution $(\rho_n, u_n, Q_n)$
on a short time interval $[0, T_n], T_n \leq T$ to the problem
\eqref{density equ Galerkin}, \eqref{order parameter Galerkin},
\eqref{var prob for u}, with initial and boundary conditions
\eqref{IC 1 Gakerkin}-\eqref{BC 2 Galerkin}.

Now we shall extend the local existence time $T_n$ to $T$. First we
can derive an energy law in a similar manner as Proposition
\ref{proposition on basic energy law}, \bea
&&\frac{d}{dt}\int_{U}\Big[\frac{\rho_n|u_n|^2}{2}+\frac{\rho_n^\gamma}{\gamma-1}+\frac{\delta\rho_n^\beta}{\beta-1}
+\mathcal{G}(Q_n)\Big]dx+\int_{U}\left(\nu|\nabla{u}|^2+(\nu+\lambda)|\mbox{div}\,{u}|^2+\Gamma
tr^2(\mathcal{H}_n)\right)dx \non\\
&&+\varepsilon\int_{U}\big(\gamma\rho_n^{\gamma-2}+\delta\beta\rho_n^{\beta-2}
\big)|\nabla\rho_n|^2dx \leq 0, \ \ \ \ \ \forall \, t \mbox{ on }
(0, T_n). \label{basic energy law 2} \eea Consequently, combined
with Lemma \ref{lemma on lower bound of total energy}, we have \be
\int_0^{T_n}\|\nabla{u}_n\|^2dt \leq \frac{2}{\nu}E_\delta[\rho_0,
q_0, Q_0], \non
 \ee
with \be E_\delta[\rho_0, q_0,
Q_0]\doteq\int_{U}\big(\frac{|q_0|^2}{2\rho_0}+\frac{\rho_0^\gamma}{\gamma-1}+\frac{\delta\rho_0^\beta}{\beta-1}
+\mathcal{G}(Q_0)\big)\,dx+\frac{(b^2-ca)^2}{2c^3}|U|. \label{delta
energy bound} \ee Meanwhile, since the $L^2$ norm and $H^2$ norm are
equivalent on each finite dimensional space $X_n$, we can deduce
from Lemma \ref{lemma on Neumann pb for density} that there exists
$C_2=C_2(n, \rho_0, q_0, Q_0, a, b, c, U)$, such that \be 0 < C_2
\leq \rho_n(t, x) \leq \frac{1}{C_2}, \ \ \ \forall \, t \in (0,
T_n), \ x \in U. \non \ee Therefore, using the energy inequality
\eqref{basic energy law 2} again, we know \be
\|u_n(t)\|_{L^\infty(U)}+\|\nabla{u}_n(t)\|_{L^\infty(U)} \leq
C_3=C_3(n, \rho_0, q_0, Q_0, a, b, c, U), \ \ \ \forall \, t \in [0,
T_n], \non \ee which allows us to extend the existence interval $(0,
T_n)$ of $u_n$ to [0, T]. Further, we know from Lemma \ref{lemma on
Neumann pb for density} and Lemma \ref{lemma on solvability of Q}
that the local solution $Q_n$ and $\rho_n$ can also be extended up
to $T$.

To finish this subsection, we summarize all the results in the
following lemma, part of which is based on \eqref{basic energy law
2}, arguments in Lemma \ref{lemma on lower bound of total energy}
and Corollary \ref{remarkofQ}, while \eqref{L1 est 5} and \eqref{L1
est 6} are due to interpolation inequalities (see \cite{FNP01} for
details).
\begin{lemma} \label{lemma on estimate for level 1 appro}
Suppose $\beta \geq 4$, there exists solution $(\rho_n, u_n, Q_n)$
to \eqref{density equ Galerkin}, \eqref{var prob for u},
\eqref{order parameter 2} in $(0, T) \times U$, and \bea
\displaystyle\sup_{t\in
[0, T]}\|\rho_n(t)\|_{L^\gamma(U)}^\gamma &\leq& C(E_\delta[\rho_0,q_0,Q_0], \gamma),  \label{L1 est 1}\\
\delta\displaystyle\sup_{t\in [0,
T]}\|\rho_n(t)\|_{L^\beta(U)}^\beta &\leq& C(E_\delta[\rho_0,q_0,Q_0], \beta),  \label{L1 est 2}\\
\displaystyle\sup_{t\in [0,
T]}\big\|\sqrt{\rho_n}(t)u_n(t)\big\|_{L^2(U)}^2 &\leq& 2E_\delta[\rho_0,q_0,Q_0],  \label{L1 est 3} \\
\|u_n\|_{L^2(0, T; H^1_0(U))} &\leq& C(E_\delta[\rho_0,q_0,Q_0], \lambda, \nu), \label{L1 est 4}  \\
\|\rho_n\|_{L^{\beta+1}((0, T)\times U)} &\leq&
C(E_\delta[\rho_0,q_0,Q_0], \varepsilon, \delta, U),
\label{L1 est 5}  \\
\varepsilon\|\nabla\rho_n\|_{L^2(0, T; L^2(U))}^2 &\leq&
C(E_\delta[\rho_0,q_0,Q_0], \beta, \delta, U, T),
\label{L1 est 6} \\
\|Q_n\|_{L^{10}((0, T)\times U)} &\leq&  C(E_\delta[\rho_0,q_0,Q_0],
a, b,c, L, \Gamma, U, T), \label{L1 est 7} \\
\|Q_n\|_{L^{\infty}(0, T; H^1(U))} &\leq&
\frac{2}{L}E_\delta[\rho_0,q_0,Q_0],
\label{L1 est 8} \\
\|\nabla{Q}_n\|_{L^{\frac{10}{3}}((0, T)\times U)} &\leq&
C(E_\delta[\rho_0,q_0,Q_0],
a, b,c, L, \Gamma, U, T), \label{L1 est 10} \\
\|Q_n\|_{L^2(0, T; H^2(U))} &\leq&  C(E_\delta[\rho_0,q_0,Q_0], a,
b, c, L, \Gamma, U, T). \label{L1 est 9},
 \eea

\end{lemma}

%%%%%%%%%%%%%%%%%%%%%%%%%%%%%%%%%%%%%%%%%%%%%%%%%%%%%%%%%%%%%%%%%%%%%%%%%%%%%%%%%%%%%%%%%%%%%%%%
\subsection{Passing to limit}
Now we shall employ the estimate in Lemma \ref{lemma on estimate for
level 1 appro} to pass to the limit as $n \rightarrow \infty$ of the
solution sequence $(\rho_n, u_n, Q_n)$ to obtain a solution to the
problem \eqref{density equ Galerkin}-\eqref{BC 2 Galerkin}. To this
end, we have to ensure that all these a priori estimates are
independent of $n$. Here and after, for the sake of convenience, we
do not distinguish sequence convergence and subsequence convergence.

To begin with, it follows from \cite{FNP01} that if $\beta > 4$,
$\gamma > \frac32$, \bea &&u_n \rightarrow u \ \  \mbox{weakly in }
L^2(0, T; H_0^1(U, \mathbb{R}^3)), \label{weak convergence 1}
\\
&&\rho_n \rightarrow \rho \ \ \mbox{in } L^4((0, T)\times U), \label{weak convergence 1-1}  \\
&&\rho_n^\gamma \rightarrow \rho^\gamma, \ \
 \rho_n^\beta \rightarrow
\rho^\beta \ \mbox{in } L^1((0, T) \times U);
 \eea
Meanwhile, using Sobolev inequality, we deduce from \eqref{L1 est
7}-\eqref{L1 est 9} that \bea
\Big\|\frac{\partial{Q_n}}{\partial{t}}\Big\|_{L^{\frac32}(U)}
&\leq&
\|u_n\|_{L^6(U)}\|\nabla{Q}_n\|_{L^2(U)}+2\|\nabla{u}_n\|_{L^{2}(U)}\|Q_n\|_{L^{6}(U)}
+\Gamma\|\mathcal{H}_n\|_{L^{\frac32}(U)} \non\\
&\leq&C\|\nabla{u}_n\|+C\|\mathcal{H}_n\|, \eea which infers that
\[ \Big\|\frac{\partial{Q}_n}{\partial{t}}\Big\|_{L^2(0, T; L^{\frac32}(U))} \leq
 C(E_\delta[\rho_0,q_0,Q_0], a,
b, c, L, \Gamma, \lambda, \nu, U, T).  \] Combined with \eqref{L1
est 9}, we know from the well-known Aubin-Lions compactness theorem
that
\[ \{Q_n\} \mbox{ is precompact in } L^2(0, T; H^1(U)). \]
Therefore, we conclude that \be Q_n \rightarrow Q \ \mbox{weakly in
} L^2(0, T; H^2(U)), \ \mbox{strongly in} \ L^2(0, T; H^1(U)).
\non\ee Hence it is easy to show that $Q$ is a weak solution to
\eqref{order parameter Galerkin}. Furthermore, we get from \eqref{L1
est 1}, \eqref{L1 est 3} and \eqref{L1 est 4} that $\{\rho_nu_n\}$
is uniformly bounded in $L^\infty(0, T;
L^{\frac{2\gamma}{\gamma+1}}(U))$. Consequently, using \eqref{weak
convergence 1} and \eqref{weak convergence 1-1}, we have
%%%%%%%%%%%%%%%%%%%%%%%%%%%%%%%%%%%%%%%%%%%%%%%%%%%%%%%%%%%%%%%%%%%%%%%%%%%%%%%%%%%%%%%%%%%%%%%%%
 \be \rho_nu_n \rightarrow \rho{u} \ \ \mbox{weakly star in }
 L^\infty(0, T; L^{\frac{2\gamma}{\gamma+1}}(U)),
  \ee
then we can pass to limit in the continuity equation \eqref{density
equ Galerkin}.

Finally, in order the prove the limit $u$ satisfies equation
\eqref{navier-stokes Galerkin}, we need the following lemma in
\cite{FNP01}.
%%%%%%%%%%%%%%%%%%%%%%%%%%%%%%%%%%%%%%%%%%%%%%%%%%%%%%%%%%%%%%%%%%%%%%%%%%%%%%%%%%%%%
\bl \label{lemma on density regularity} There exist $r>1$, $s>2$
such that $\partial_t\rho_n$, $\Delta\rho_n$ are uniformly bounded
in $L^r((0, T)\times U)$, $\nabla\rho_n$ is uniformly bounded in
$L^s(0, T)\times U)$. And the limit function $\rho$ satisfies
equation \eqref{density equ Galerkin} almost everywhere on $(0, T)
\times U$ and the boundary condition \eqref{BC 1 Galerkin} in the
trace sense. \el

We now show that for any fixed test function $\psi$ in \eqref{var
prob for u}, $\int_{U}\langle \rho_n{u}_n(t), \psi \rangle \,dx$ is
equi-continuous in $t$. By Lemma \ref{lemma on estimate for level 1
appro} and Lemma \ref{lemma on density regularity}, we get for any
$0<\zeta<1$, it holds \bea &&\Big|\int_t^{t+\zeta}\int_{U}\langle
L\nabla{Q}_n\odot\nabla{Q}_n-\mathcal{F}(Q_n)I_3-LQ_n\mathcal{H}(Q_n)+L\mathcal{H}(Q_n)Q_n,
\nabla\psi \rangle\, dxds \Big| \non\\
&\leq&C\zeta\|\nabla\psi\|_{L^\infty(U)}\displaystyle\sup_{0 \leq
t\leq T}(\|\nabla{Q}_n(t)\|^2+1)+C\int_{t}^{t+\zeta}\|Q_n\|\|\Delta{Q}_n\|\|\nabla\psi\|_{L^{\infty}(U)}\,ds \non \\
&\leq&C\zeta+C\|\nabla
\psi\|_{L^{\infty}(U)}\Big(\int_{t}^{t+\zeta}\|\Delta{Q}_n\|^2\,ds\Big)^{\frac12}\Big(\int_{t}^{t+\zeta}\|Q_n\|^2ds\Big)^{\frac12}
   \non\\
&\leq&C\zeta^{\frac12}.  \non\eea \be
\Big|\int_t^{t+\zeta}\int_{U}\langle
\rho_n^{\gamma}+\delta\rho_n^\beta, \psi \rangle \,dxdt \Big| \leq
C\zeta\|\nabla\psi\|_{L^\infty(U)}\displaystyle\sup_{0 \leq
t\leq T}\int_{U}\left(\rho_n^\gamma+\delta\rho_n^\beta \right)\,dx \leq C\zeta. \non\\
\ee \be \Big|\int_t^{t+\zeta}\int_{U}\langle \mathcal{L}u_n, \psi
\rangle\, dxds \Big| \leq
C\zeta\|\nabla\psi\|_{L^2(U)}\Big(\int_0^{T}\|\nabla{u}_n(t)\|^2\,dt\Big)^\frac12
\leq C\zeta. \non\ee \be \Big|\int_t^{t+\zeta}\int_{U}\langle
\nabla\cdot(\rho_nu_n\otimes u_n), \psi \rangle\, dxds \Big| \leq
C\zeta\|\nabla\psi\|_{L^\infty}\displaystyle\sup_{0 \leq t\leq
T}\int_{U}\rho_n|u_n|^2dx \leq C\zeta.  \non\ee \bea
&&\Big|\int_t^{t+\zeta}\int_{U}\langle
\varepsilon\nabla\rho_n\cdot\nabla{u}_n, \psi \rangle \,dxds \Big|
\non\\
&\leq&\varepsilon\|\psi\|_{L^\infty(U)}\Big(\int_t^{t+\zeta}\int_{U}|\nabla{u}_n|^2dt\Big)^\frac12
\Big(\int_t^{t+\zeta}\big(\int_{U}|\nabla{\rho}_n|^2\big)^{\frac{s}{2}}\Big)^{\frac{1}{s}}\zeta^{\frac12-\frac{1}{s}}
\non\\
&\leq&C\zeta^{\frac12-\frac{1}{s}}, \non
 \eea
where we used Lemma \ref{lemma on density regularity} for the last
estimate. Hence we know (c.f. Corollary 2.1 in \cite{F04}) \be
\rho_nu_n \rightarrow \rho{u} \ \mbox{ in } \ C\big([0, T];
L_{weak}^{\frac{2\gamma}{\gamma+1}}\big). \label{convergence in weak
star} \ee Due to the compact embedding
$L^\frac{2\gamma}{\gamma+1}(U) \hookrightarrow H^{-1}(U)$ if $\gamma
> \frac32$, we infer from \eqref{convergence in weak star} that
\[ \rho_nu_n \rightarrow \rho{u} \ \mbox{ in } C([0, T]; H^{-1}(U)), \]
which together with \eqref{weak convergence 1} indicates
\[ \rho_nu_n\otimes{u}_n \rightarrow \rho{u}\otimes{u} \ \ \mbox{in } \mathcal{D}'((0, T)\times U).
\] Finally,
the convergence of the remaining term $\nabla\rho_n\cdot\nabla{u}_n
\rightarrow \nabla\rho\cdot\nabla{u}$ in $\mathcal{D}'((0, T)\times
U)$ follows \cite{FNP01}.

In all, we summarize the above results as follows.
\begin{proposition} \label{proposition on level 1 existence}
The problem \eqref{density equ Galerkin}-\eqref{BC 2 Galerkin}
admits a weak solution $(\rho, u, Q)$ which satisfies all estimates
in Lemma \ref{lemma on estimate for level 1 appro}. Moreover, the
energy inequality \eqref{basic energy law 2} holds in
$\mathcal{D}'(0, T)$ and there exists $r>1$, such that $\rho_t,
\Delta\rho \in L^r((0, T)\times U)$ and the equation \eqref{density
equ Galerkin} is satisfied pointwisely in $(0, T)\times U$. In
addition, $Q \in S_0^3$ a.e. in $[0, T] \times U$.

\end{proposition}

%%%%%%%%%%%%%%%%%%%%%%%%%%%%%%%%%%%%%%%%%%%%%%%%%%%%%%%%%%%%%%%%%%%%%%%%%%%%%%%%%%%%%%%%%%%%%%%%%%%%%%%%%%%%
\section{Vanishing artificial viscosity}\setcounter{equation}{0}
Our next aim is to let $\varepsilon \rightarrow 0$ in the modified
continuity equation \eqref{density equ Galerkin} and velocity
equation \eqref{navier-stokes Galerkin} for passing to the limit. We
denote by $(\rho_{\varepsilon}, u_{\varepsilon}, Q_\varepsilon)$ the
corresponding solution of the problem \eqref{density equ
Galerkin}-\eqref{BC 2 Galerkin}. At this point, we are lack in the
bound of $\nabla\rho_\varepsilon$ (see \eqref{L1 est 6}) and
consequently, it is essential for the study of strong compactness of
$\{\rho_\varepsilon\}_{\varepsilon > 0}$ in $L^1((0, T)\times U)$.

\subsection{Density estimates independent of viscosity}
To begin with, we deduce from \eqref{L1 est 4} and \eqref{L1 est 6}
that \bea
\varepsilon\nabla\rho_\varepsilon\cdot\nabla{u_\varepsilon}
\rightarrow 0 \ \mbox{in } L^1((0, T)\times U), \label{var conv 1}\\
\varepsilon\Delta\rho_\varepsilon \rightarrow 0 \ \mbox{in } L^2(0,
T; H^{-1}(U)). \label{var conv 2}\eea And in the same way as last
section, we get \be Q_\varepsilon \rightarrow Q \ \mbox{weakly in }
L^2(0, T; H^2(U)) \ \mbox{and strongly in } L^2(0, T; H^1(U)).
\label{var conv 3} \ee \br Since $Q_\varepsilon \in S_0^3$ a.e. in
$[0, T]\times U$, it is also true that its limit $Q \in S_0^3$ a.e.
in $[0, T]\times U$ because of the above convergence result
\eqref{var conv 3}. \er

More importantly, we can prove the following estimate of density
independent of $\varepsilon$.
\begin{lemma}\label{lemma on density estimate without varepsilon}
Suppose $(\rho_\varepsilon, u_\varepsilon, Q_\varepsilon)$ is a
sequence of solutions to the problem \eqref{density equ
Galerkin}-\eqref{BC 2 Galerkin} constructed in Proposition
\ref{proposition on level 1 existence}. Then \be
\|\rho_\varepsilon\|_{L^{\gamma+1}((0, T)\times
U)}+\|\rho_\varepsilon\|_{L^{\beta+1}((0, T)\times U)} \leq
C(E_\delta(\rho_0,q_0,Q_0), a, b, c, \delta, \beta, \lambda, \nu, L,
U, T). \ee
\end{lemma}
\begin{proof}
The proof is similar to \cite{FNP01} (c.f. Lemma 3.1). We introduce
an operator (\cite{B80, GA94}) $$\mathcal{B}: \{f \in L^p(U):
\int_{U}f\,dx =0 \}\mapsto [H_0^{1, p}(U)]^3,$$ such that
$v=\mathcal{B}(f)$ solves the following problem
\[ \mbox{div}\,v = f \ \mbox{in } U, \ v|_{\partial U}=0. \]
Then we take the test function for \eqref{navier-stokes Galerkin} as
\[ \psi(t)\mathcal{B}(\rho_\varepsilon-m_0), \ \psi \in \mathcal{D}(0, T), \
0 \leq \psi \leq 1, \ m_0=\frac{1}{|U|}\int_{U}\rho(t)dx.  \] We
note that the total mass $m_0$ is a constant such that the test
function is well defined. Then direct calculations lead to
\begin{eqnarray*}
&&\int_0^{T}\int_{U}\psi(\rho_\varepsilon^{\gamma+1}+\delta\rho_\varepsilon^{\beta+1})dxdt\\
&=&m_0\int_{0}^T\psi\Big(\int_{U}\rho_\varepsilon^{\gamma}+\delta\rho_\varepsilon^{\beta}\,dx\Big)dt
+(\lambda+\nu)\int_{0}^T\psi\int_{U}\rho_\varepsilon\mbox{div}u_\varepsilon\,dxdt\\
&&-\int_{0}^T\psi_t\int_{U}\big\langle\rho_\varepsilon
u_\varepsilon,
\mathcal{B}(\rho_\varepsilon-m_0)\big\rangle\,dxdt+\nu\int_{0}^T\psi\int_{U}\nabla
u_\varepsilon : \nabla\mathcal{B}(\rho_\varepsilon-m_0)\,dxdt\\
&&-\int_{0}^T\psi\int_{U}\rho_\varepsilon{u}_\varepsilon\otimes{u}_\varepsilon
:\nabla\mathcal{B}(\rho_\varepsilon-m_0)\,dxdt-\varepsilon\int_{0}^T\psi\int_{U}\big\langle\rho_\varepsilon{u}_\varepsilon
, \mathcal{B}(\Delta\rho_\varepsilon)\big\rangle\,dxdt\\
&&-\int_{0}^T\psi\int_{U}\big\langle\rho_\varepsilon{u}_\varepsilon
,\mathcal{B}(\mbox{div}(\rho_\varepsilon{u}_\varepsilon))\big\rangle\,dxdt
+\varepsilon\int_{0}^T\psi\int_{U}\nabla{u}_\varepsilon :
\mathcal{B}(\rho_\varepsilon-m_0)\nabla\rho_\varepsilon\,dxdt\\
&&+\int_{0}^T\psi\int_{U}\big(\nabla{Q}_\varepsilon\otimes\nabla{Q}_\varepsilon
-\mathcal{F}(Q_\varepsilon)I_3\big)
:\nabla\mathcal{B}(\rho_\varepsilon-m_0)\,dxdt\\
&&-L\int_{0}^T\psi\int_{U}\big(Q_\varepsilon\mathcal{H}({Q}_\varepsilon)-
\mathcal{H}({Q}_\varepsilon){Q}_\varepsilon\big):
\nabla\mathcal{B}(\rho_\varepsilon-m_0)\,dxdt\\
&\doteq&I_1+\cdots+I_{10}.
\end{eqnarray*}
Now we estimate $I_1, \cdots, I_{10}$. By \eqref{L1 est 1},
\eqref{L1 est 2}, \eqref{L1 est 4} and \eqref{L1 est 5}, we get
\[|I_1| \leq |m_0|T\Big(\displaystyle\sup_{t\in [0,T]}\|\rho_\varepsilon(t)\|_{L^\gamma(U)}^\gamma
+\delta\displaystyle\sup_{t\in [0,
T]}\|\rho_n(t)\|_{L^\beta(U)}^\beta\Big) \leq
C(E_\delta(\rho_0,q_0,Q_0), \gamma, \beta, T).
\]
%%%%%%%%%%%%%%%%%%%%%%%%%%%%%%%%%%%%%%%%%%%%%%%%%%%%%%%%%%%%%%%%%%%%%%%%%%%%%%%%%
\[|I_2| \leq (\lambda+\nu)\|\rho_\varepsilon\|_{L^2(0, T; L^2(U))}\|\nabla{u}_\varepsilon\|_{L^2(0, T; L^2(U))}
\leq C(E_\delta(\rho_0,q_0,Q_0), \delta, \gamma, \beta, \lambda,
\nu, U).
\]
By the property of the operator $\mathcal{B}$, we know
$$\|\mathcal{B}(\rho_\varepsilon-m_0)\|_{H_0^{1, \beta}(U)} \leq
C(\beta, U)\|\rho_\varepsilon-m_0\|_{L^\beta(U)}.$$ Using Sobolev
embedding theorem for $\beta > 4$, \eqref{L1 est 2} and \eqref{L1
est 3}, we get
%%%%%%%%%%%%%%%%%%%%%%%%%%%%%%%%%%%%%%%%%%%%%%%%%%%%%%%%%%%%%%%%%%%%%%%%%%%%%%%%%
\be |I_3| \leq
C\int_0^T\|\sqrt{\rho_\varepsilon}\|_{L^2(U)}\|\sqrt{\rho_\varepsilon}u_\varepsilon\|_{L^2(U)}
\|\mathcal{B}(\rho_\varepsilon-m_0)\|_{L^\infty(U)} \leq
C(E_\delta(\rho_0,q_0,Q_0), \delta, \gamma, \beta, U). \non \ee
Similar to the estimate for $I_3$, it holds
\[ |I_4| \leq C\|u_\varepsilon\|_{L^2(0, T; H^1(U))}\|\rho_\varepsilon\|_{L^2(0,T; L^2(U))}
\leq C(E_\delta(\rho_0,q_0,Q_0), \delta, \beta, \lambda, \nu, U).
\]
\[|I_5| \leq \int_0^T\|\rho_\varepsilon\|_{L^3(U)}\|u_\varepsilon\|_{L^6(U)}^2\|\rho_\varepsilon\|_{L^3(U)}\,dt
\leq C(E_\delta(\rho_0,q_0,Q_0), \delta, \beta, \lambda, \nu, U).\]
And by \eqref{L1 est 4}, \eqref{L1 est 5}, \eqref{L1 est 6}, the
property of operator $\mathcal{B}$ and Sobolev embedding theorem, it
yields  \[ |I_6| \leq
C\varepsilon\int_0^{T}\|\rho_\varepsilon\|_{L^3(U)}\|u_\varepsilon\|_{L^6(U)}\|\nabla\rho_\varepsilon\|_{L^2(U)}\,dt
\leq C(E_\delta(\rho_0,q_0,Q_0), \delta, \beta, \lambda, \nu, U, T),
\ \ \mbox{for } \varepsilon \leq 1.
\]
Next, since the operator $\mathcal{B}$ enjoys the property
$$\|\mathcal{B}(f)\|_{L^2(U)} \leq C(U)\|g\|_{L^2(U)}\text{ for }\,\,
\mathcal{B}(f)=\mbox{div}\,g$$ with $$g\cdot \vec{n}|_{\partial
U}=0,$$ we infer from \eqref{L1 est 4} and \eqref{L1 est 5} that
\bea |I_7| &\leq&
\int_0^T\|\rho_\varepsilon\|_{L^3(U)}\|u_\varepsilon\|_{L^6(U)}\|\rho_\varepsilon{u}_\varepsilon\|_{L^2(U)}\,dt
\leq
\int_0^T\|\rho_\varepsilon\|_{L^3(U)}^2\|\nabla{u}_\varepsilon\|_{L^2(U)}^2\,dt
\non\\
&\leq& C(E_\delta(\rho_0,q_0,Q_0), \delta, \beta, \lambda, \nu, U).
\non \eea Further, by \eqref{L1 est 4} and \eqref{L1 est 6}, we
obtain \bea |I_8| &\leq&
\sqrt{\varepsilon}\|\sqrt{\varepsilon}\nabla\rho_\varepsilon\|_{L^2(0,
T; L^2(U))}\|\nabla{u}_\varepsilon\|_{L^2(0, T;
L^2(U))}\|\mathcal{B}(\rho_\varepsilon-m_0)\|_{L^\infty(0, T;
L^\infty(U))} \non\\
&\leq&C(\beta,U)\sqrt{\varepsilon}\|\sqrt{\varepsilon}\nabla\rho_\varepsilon\|_{L^2(0,
T; L^2(U))}\|\nabla{u}_\varepsilon\|_{L^2(0, T;
L^2(U))}\|\rho_\varepsilon\|_{L^\infty(0, T; L^\beta(U))} \non\\
&\leq& C(E_\delta(\rho_0,q_0,Q_0), \delta, \beta, \lambda, \nu, U,
T), \ \ \mbox{for } \varepsilon \leq 1. \eea Then by \eqref{L1 est
2}, \eqref{L1 est 7}, and \eqref{L1 est 10}, we know
%%%%%%%%%%%%%%%%%%%%%%%%%%%%%%%%%%%%%%%%%%%%%%%%%%%%%%%%%%%%%%%%%%%%%%%%%%%%%%%%%%%%%%%%%%%%%
 \bea |I_9|&\leq&\int_0^{T}\|\nabla{Q}_\varepsilon\|_{L^{\frac{10}{3}}(U)}^2\|\nabla\mathcal{B}(\rho_\varepsilon-m_0)\|_{L^{\frac52}(U)}
+\|\mathcal{F}(Q_\varepsilon)\|_{L^{\frac{5}{2}}(U)}\|\nabla\mathcal{B}(\rho_\varepsilon-m_0)\|_{L^{\frac53}(U)}\,dt  \non\\
 &\leq& C(L,
 U)\int_0^{T}\big(\|\nabla{Q}_\varepsilon\|_{L^{\frac{10}{3}}(U)}^2+1\big)\|\nabla\mathcal{B}(\rho_\varepsilon-m_0)\|_{L^{\frac52}(U)}
\,dt \non\\
&&+C(a,b,c)\int_0^T \big(\|tr^2(Q^2)\|_{L^{\frac52}(U)}+1)\|\nabla\mathcal{B}(\rho_\varepsilon-m_0)\|_{L^{\frac53}(U)}\,dt      \non\\
 &\leq& C(E_\delta(\rho_0,q_0,Q_0), a, b, c, \delta, \beta, \lambda, \nu, L, U,
T).\non \eea Finally, we deduce from \eqref{L1 est 2}, \eqref{L1 est
7} and \eqref{L1 est 9} \bea |I_{10}| &\leq&
2L\int_0^{T}\|Q_\varepsilon\|_{L^{10}(U)}\|\Delta{Q}_\varepsilon\|_{L^2(U)} \|\nabla\mathcal{B}(\rho_\varepsilon-m_0)\|_{L^{\frac{5}{2}}(U)}\,dt \non\\
&\leq& C(E_\delta(\rho_0,q_0,Q_0), a, b, c, \delta, \beta, \lambda,
\nu, L, U, T). \non \eea Hence we finish the proof by summing up all
previous results for $I_1, \cdots, I_{10}$.

\end{proof}
Lemma \ref{lemma on density estimate without varepsilon} together
with \eqref{L1 est 2} imply that \be \rho_\varepsilon \rightarrow
\rho \ \mbox{in } C(0, T; L_{weak}^\beta(U)) \ \mbox{and weakly in }
L^{\beta+1}((0, T)\times U). \label{var conv 4} \ee Moreover, \be
u_\varepsilon \rightarrow u \ \mbox{weakly in } L^2(0, T; H_0^1(U)),
\label{var conv 5} \ee which together with \eqref{L1 est 1} and
\eqref{var conv 4} yield
%%%%%%%%%%%%%%%%%%%%%%%%%%%%%%%%%%%%%%%%%%%%%%%%%%%%%%%%%%%%%%%%%%%%%%%%%%%%
 \be \rho_\varepsilon{u}_\varepsilon \rightarrow \rho{u} \ \mbox{in } C\big([0, T]; L_{weak}^{\frac{2\gamma}{\gamma+1}}(U)\big).
 \label{var conv 6}  \ee
Applying the same arguments as in the last section, and noting that
$\frac{2\gamma}{\gamma+1} > \frac{6}{5}$, it then follows from
\eqref{var conv 5}  and \eqref{var conv 6} that \bea
\rho_\varepsilon{u}_\epsilon\otimes{u}_\varepsilon \rightarrow
\rho{u}\otimes{u} \ \mbox{in } \mathcal{D}'((0, T)\times U).
\label{var conv 7} \eea Meanwhile, \eqref{var conv 3} implies that
\bea &&-\nabla\cdot(\nabla{Q}_\varepsilon
\odot\nabla{Q}_\varepsilon-\mathcal{F}(Q_\varepsilon)I_3)
+L\nabla\cdot(Q_\varepsilon\mathcal{H}({Q}_\varepsilon)-\mathcal{H}\big(Q_\varepsilon){Q}_\varepsilon\big)
\non\\
&\rightarrow&
-\nabla\cdot(\nabla{Q}\odot\nabla{Q}-\mathcal{F}(Q)I_3)+L\nabla\cdot\big(Q\mathcal{H}(Q)-\mathcal{H}(Q){Q}\big)
\ \ \mbox{in } \mathcal{D}'((0, T)\times U). \label{var conv 8b}
 \eea
In conclusion, we prove the limit $(\rho, u, Q)$ satisfies the
following equations in $\mathcal{D}'((0, T)\times U)$:
\bea \rho_t+\nabla\cdot(\rho{u})&=&0, \label{density equ viscosity} \\
(\rho{u})_t+\nabla\cdot(\rho{u}\otimes{u})+\nabla{p}&=&\mathcal{L}u-\nabla\cdot\big(L\nabla{Q}\odot\nabla{Q}
-\mathcal{F}(Q)I_3\big) \non\\
&&+L\nabla\cdot\big(Q\mathcal{H}(Q)-\mathcal{H}(Q)Q\big),
 \label{navier-stokes viscosity} \\
Q_t+u\cdot\nabla{Q}-\Omega{Q}+Q\Omega&=&\Gamma\mathcal{H}(Q),
 \label{order parameter viscosity} \eea
with the initial data
\[  \rho(0)=\rho_0, \ \ (\rho{u})(0)=q_0, \ \ Q(0)=Q_0. \]

\br Using Lemma \ref{lemma on density estimate without varepsilon}
and the assumption $\beta > \gamma$ we know the pressure $p$ in the
above system \eqref{density equ viscosity}-\eqref{order parameter
viscosity} has the property  \be
\rho_\varepsilon^\gamma+\delta\rho_\varepsilon^{\beta} \rightarrow p
\ \mbox{weakly in } L^{\frac{\beta+1}{\beta}}((0, T)\times U).
\label{var conv 9}\ee \er

The remaining part of this section is to improve the convergence in
\eqref{var conv 9} to be strong in $L^1((0, T)\times U)$, such that
\[ p=\rho^\gamma+\delta\rho^\beta. \]

%%%%%%%%%%%%%%%%%%%%%%%%%%%%%%%%%%%%%%%%%%%%%%%%%%%%%%%%%%%%%%%%%%%%%%%%%%%%%%%%%%%%%%%%%%%%%%%%%%%%
\subsection{The effective viscous flux}
The quantity
$\rho^\gamma+\delta\rho^\beta-(\lambda+2\nu)\mbox{div}\,u$ is
usually referred to as the effective viscous flux. We shall find
that it plays an essential role on our coupled system (see also
\cite{H95, L98}).

\bl \label{lemma on viscous flux} Let $(\rho_\varepsilon,
u_\varepsilon, Q_\varepsilon)$ be a sequence of solutions
constructed in Proposition \ref{proposition on level 1 existence},
and $(\rho, u, Q)$ be its limit satisfying \eqref{density equ
viscosity}-\eqref{order parameter viscosity}, respectively. Then for
any $\psi \in \mathcal{D}(0, T)$, $\phi \in \mathcal{D}(U)$, it
holds \bea &&\displaystyle\lim_{\varepsilon \rightarrow
0^{+}}\int_0^T\psi\int_{U}\phi\big(\rho_{\varepsilon}^\gamma+\delta\rho_{\varepsilon}^\beta
-(\lambda+2\nu)\mbox{div}\,u_{\varepsilon}
\big)\rho_\varepsilon\,dxdt \non\\
&=&\int_0^T\psi\int_{U}\phi\big(p-(\lambda+2\nu)\mbox{div}\,u\big)\rho\,dxdt
\label{limit of viscous flux} \eea \el
%%%%%%%%%%%%%%%%%%%%%%%%%%%%%%%%%%%%%%%%%%%%%%%%%%%%%%%%%%%%%%%%%%%%%%%%%%%%%%%%%%
\br
 It is worth pointing out that from the fluid mechanics point of view, the quantity
 $P-(\lambda+2\nu)\mbox{div}\,u$ appearing in \eqref{limit of viscous flux} is the
amplitude of the normal viscous stress augmented by the hydrostatic
pressure. \er

\begin{proof}
We consider the singular integral operator
\[ \mathcal{A}_i=\partial_{x_i}\Delta^{-1}, \]
or equivalently in terms of its Fourier symbol
\[ \mathcal{A}_j(\xi)=\frac{-\sqrt{-1}\xi_j}{|\xi|^2}. \]
By Proposition \ref{proposition on level 1 existence},
$\rho_\varepsilon, u_\varepsilon$ satisfy \eqref{density equ
Galerkin} a.e. on $(0, T)\times U$ with the boundary condition
\eqref{BC 1 Galerkin}. In particular, we extend $\rho_\varepsilon$,
$u_\varepsilon$ to be zero outside $U$. Then it yields \be
\partial_t\rho_\varepsilon+\mbox{div}\,(\rho_\varepsilon{u}_\varepsilon)=\varepsilon\mbox{div}(1_{U}\nabla\rho_\varepsilon)
\ \ \mbox{in } \mathcal{D}'((0, T)\times \mathbb{R}^3) \ee with
$1_{U}$ the characteristic function on $U$. Next, we consider the
vector-valued test function \[ \varphi(t,
x)=\psi(t)\phi(x)\mathcal{A}(\rho_\varepsilon)\doteq\psi(t)\phi(x)\big(\mathcal{A}_1(\rho_\varepsilon),
\mathcal{A}_2(\rho_\varepsilon), \mathcal{A}_3(\rho_\varepsilon)
\big), \ \ \mbox{where} \ \ \psi \in \mathcal{D}(0, T), \ \phi \in
\mathcal{D}(U).
\]
Analogously, after direct calculations we derive \bea
&&\int_0^T\psi\int_{U}\phi\big(\rho_{\varepsilon}^\gamma+\delta\rho_{\varepsilon}^\beta
-(\lambda+2\nu)\mbox{div}\,u_{\varepsilon}
\big)\rho_\varepsilon\,dxdt\non\\
&=&(\lambda+\nu)\int_0^T\psi\int_{U}\mbox{div}\,u_{\varepsilon}
\langle \nabla\phi,
\mathcal{A}(\rho_\varepsilon)\rangle\,dxdt-\int_0^T\psi\int_{U}(\rho_{\varepsilon}^\gamma+\delta\rho_{\varepsilon}^\beta)
\langle \nabla\phi, \mathcal{A}(\rho_\varepsilon)\rangle\,dxdt\non\\
&&-\int_0^T\psi\int_{U}\rho_{\varepsilon}u_\varepsilon\otimes
u_\varepsilon:\nabla\phi\otimes\mathcal{A}(\rho_\varepsilon)\,dxdt
-\int_0^T\psi_t\int_{U}\phi \langle \rho_{\varepsilon}u_\varepsilon,
\mathcal{A}(\rho_\varepsilon)\rangle\,dxdt \non\\
&&+\nu\int_0^T\psi\int_{U}\mathcal{A}(\rho_\varepsilon)\otimes\nabla\phi
: \nabla{u}_\varepsilon \,dxdt
-\nu\int_0^T\psi\int_{U}\nabla\mathcal{A}(\rho_\varepsilon) : {u}_\varepsilon\otimes\nabla\phi \,dxdt    \non\\
&&+\nu\int_0^T\psi\int_{U}\rho_\varepsilon({u}_\varepsilon\cdot\nabla\phi)
\,dxdt+\int_0^T\psi\int_{U}\phi\big[\rho_\varepsilon\nabla_j\mathcal{A}_i(\rho_\varepsilon{u}_\varepsilon^j)
-\rho_\varepsilon{u}_\varepsilon^j\nabla_j\mathcal{A}_i(\rho_\varepsilon)\big]\,dxdt  \non\\
&&-\int_0^T\psi\int_{U}\big(L\nabla{Q}_\varepsilon\odot\nabla{Q}_\varepsilon-\mathcal{F}(Q_\varepsilon)I_3
\big):\nabla\big(\phi\mathcal{A}(\rho_\varepsilon)\big)\,dxdt\non\\
&&+L\int_0^T\psi\int_{U}\big(Q_\varepsilon\mathcal{H}({Q}_\varepsilon)-\mathcal{H}({Q}_\varepsilon){Q}_\varepsilon
\big):
\nabla\big(\phi\mathcal{A}(\rho_\varepsilon)\big) \,dxdt   \non\\
&&-\varepsilon\int_0^T\psi\int_{U}\phi\big\langle \rho_{\varepsilon}
u_\varepsilon,\mathcal{A}(\mbox{div}\,(1_{U}\nabla\rho_\varepsilon))\big\rangle\,dxdt
+\varepsilon\int_0^T\psi\int_{U}\phi \nabla{u}_\varepsilon:
\mathcal{A}(\rho_\varepsilon)\otimes\nabla\rho_\varepsilon
\,dxdt\non\\
&\doteq&I_1+\cdots+I_{12}. \label{test fun viscosity-1} \eea In the
meantime, we can repeat the above procedures to the limit equations
\eqref{density equ viscosity} and \eqref{navier-stokes viscosity},
since we have the following result from \cite{FNP01}.
\begin{lemma} \label{lemma on 0 extension}
Suppose $\rho \in L^2((0, T)\times U)$, $u \in L^2(0, T; H_0^1(U))$
is a solution of \eqref{density equ viscosity} in $\mathcal{D}'((0,
T)\times U)$. Then the equation \eqref{density equ viscosity} still
holds in $\mathcal{D}'((0, T)\times \mathbb{R}^3)$, provided $(\rho,
u)$ are extended to be $0$ in $\mathbb{R}^3\backslash U$.
\end{lemma}
Consequently, the counterpart to \eqref{test fun viscosity-1} is
\bea &&\int_0^T\psi\int_{U}\phi(p -(\lambda+2\nu)\mbox{div}\,u)\rho\,dxdt\non\\
&=&(\lambda+\nu)\int_0^T\psi\int_{U}\mbox{div}\,u \langle
\nabla\phi, \mathcal{A}(\rho)\rangle\,dxdt-\int_0^T\psi\int_{U}P
\langle \nabla\phi, \mathcal{A}(\rho)\rangle\,dxdt\non\\
&&-\int_0^T\psi\int_{U}\rho{u}\otimes{u}:\nabla\phi\otimes\mathcal{A}(\rho)\,dxdt
-\int_0^T\psi_t\int_{U}\phi \langle \rho{u},
\mathcal{A}(\rho)\rangle\,dxdt \non\\
&&+\nu\int_0^T\psi\int_{U}\mathcal{A}(\rho)\otimes\nabla\phi :
\nabla{u} \,dxdt
-\nu\int_0^T\psi\int_{U}\nabla\mathcal{A}(\rho) : {u}\otimes\nabla\phi \,dxdt    \non\\
&&+\nu\int_0^T\psi\int_{U}\rho(u\cdot\nabla\phi) \,dxdt
+\int_0^T\psi\int_{U}\phi\big[\rho\nabla_j\mathcal{A}_i(\rho{u}^j)
-\rho{u}^j\nabla_j\mathcal{A}_i(\rho)\big]\,dxdt  \non\\
&&-\int_0^T\psi\int_{U}\big(L\nabla{Q}\odot\nabla{Q}-\mathcal{F}(Q)I_3
\big):\nabla\big(\phi\mathcal{A}(\rho)\big)\,dxdt\non\\
&&+L\int_0^T\psi\int_{U}\big(Q\mathcal{H}(Q)-\mathcal{H}(Q){Q}
\big):
\nabla\big(\phi\mathcal{A}(\rho)\big) \,dxdt   \non\\
&\doteq&J_1+\cdots+J_{10}. \label{test fun viscosity-2} \eea

Due to the classical $L^p$-theory for elliptic problems, we have
%%%%%%%%%%%%%%%%%%%%%%%%%%%%%%%%%%%%%%%%%%%%%%%%%%%%%%%%%%%%%%%%%%%%%%%%%%%
 \be \|\mathcal{A}(v)\|_{H^{1,s}(U)} \leq C(s, U)\|v\|_{L^s(\mathbb{R}^3)}, \ \ \ 1 < s <
 \infty,
 \label{Mikhlin Multiplier theorem}     \ee
which combined with \eqref{var conv 4} lead to \be
\mathcal{A}(\rho_\varepsilon) \rightarrow \mathcal{A}(\rho) \ \
\mbox{in } C\big(\overline{(0, T)\times U } \big),
 \label{var conv 8} \ee and henceforth
\be \nabla\mathcal{A}(\rho_\varepsilon) \rightarrow
\nabla\mathcal{A}(\rho) \ \ \mbox{in } C\big([0, T];
L^\beta_{weak}(U) \big).  \label{var conv 9}
 \ee   Therefore, direct derivations
from \eqref{var conv 5} and \eqref{var conv 8} show  that
\[ I_1 \rightarrow J_1, \ \ I_5 \rightarrow J_5,  \ \ \mbox{as } \ \varepsilon \rightarrow 0.   \]
Meanwhile, \eqref{var conv 6} and \eqref{var conv 8} indicate that
\[ I_4 \rightarrow J_4,  \ \ \mbox{as } \ \varepsilon \rightarrow 0.   \]
By \eqref{var conv 5} and \eqref{var conv 6}, we know
$\rho_\varepsilon{u}_\varepsilon \otimes{u}_\varepsilon \in
L^2\big(0, T; L^{\frac{6\gamma}{3+4\gamma}}(U)\big) $. Then it
infers from \eqref{var conv 7} that
\[ \rho_\varepsilon{u}_\varepsilon \otimes{u}_\varepsilon \rightarrow \rho{u}\otimes{u} \ \ \mbox{weakly in }
L^2\big(0, T; L^{\frac{6\gamma}{3+4\gamma}}(U)\big). \]
Consequently, we infer from \eqref{var conv 8} that
\[ I_3 \rightarrow J_3,  \ \ \mbox{as } \ \varepsilon \rightarrow 0,   \]
provided $\beta \geq \frac{6\gamma}{2\gamma-3}$. Note that
\eqref{var conv 9} indicates $\nabla\mathcal{A}(\rho_\varepsilon)
\rightarrow \nabla\mathcal{A}(\rho)$ strongly in $C([0, T],
H^{-1}(U))$, hence we get from \eqref{var conv 5} that
\[ I_6 \rightarrow J_6,  \ \ \mbox{as } \ \varepsilon \rightarrow 0.  \]
Analogously, since $\beta > 4$, we can apply similar argument as for
$I_6$ to conclude
\[ I_7 \rightarrow J_7,  \ \ \mbox{as } \ \varepsilon \rightarrow 0.  \]
For $I_8$, it follows from \eqref{var conv 4}, \eqref{var conv 6}
and \eqref{Mikhlin Multiplier theorem} that if $\beta >
\frac{6\gamma}{2\gamma-3}$, then
\[ \rho_\varepsilon\nabla_j\mathcal{A}_i(\rho_\varepsilon{u}_\varepsilon^j)
-\rho_\varepsilon{u}_\varepsilon^j\nabla_j\mathcal{A}_i(\rho_\varepsilon)
\in L^\infty(0, T; L^\alpha(U)), \ \ \mbox{with } \
\frac{\gamma+1}{2\gamma}+\frac{1}{\beta}=\frac{1}{\alpha} <
\frac{5}{6}.
\]
Hence we infer from the celebrated Div-Curl Lemma and compact
embedding $L^\alpha(U)$ $\hookrightarrow$ $H^{-1}(U)$ that
\[  \rho_\varepsilon\nabla_j\mathcal{A}_i(\rho_\varepsilon{u}_\varepsilon^j)
-\rho_\varepsilon{u}_\varepsilon^j\nabla_j\mathcal{A}_i(\rho_\varepsilon)
\rightarrow
\rho\nabla_j\mathcal{A}_i(\rho{u})-\rho{u}\nabla_j\mathcal{A}_i(\rho)
\ \ \mbox{strongly in } H^{-1}(U).
\]
Then applying Lebesgue convergence theorem, we obtain
\[  \rho_\varepsilon\nabla_j\mathcal{A}_i(\rho_\varepsilon{u}_\varepsilon^j)
-\rho_\varepsilon{u}_\varepsilon^j\nabla_j\mathcal{A}_i(\rho_\varepsilon)
\rightarrow
\rho\nabla_j\mathcal{A}_i(\rho{u})-\rho{u}\nabla_j\mathcal{A}_i(\rho)
\ \ \mbox{strongly in } L^2\big(0, T; H^{-1}(U)\big),  \] which
combined with \eqref{var conv 5} yields
\[ I_8 \rightarrow J_8, \ \ \mbox{as } \ \varepsilon \rightarrow 0 \]
Using \eqref{L1 est 2}, \eqref{L1 est 4}, \eqref{L1 est 6} and
\eqref{Mikhlin Multiplier theorem}, we get
\[ I_{11} \rightarrow 0, \ \ I_{12} \rightarrow 0, \ \ \mbox{as } \ \varepsilon \rightarrow 0. \]
It remains to prove the corresponding convergence results for $I_9$
and $I_{10}$, which are related to the order parameter $Q$. Notice
that both $I_9$ and $J_9$ can be decomposed in the following manner:
\bea
I_{9}&=&-\int_0^T\psi\int_{U}\big(L\nabla{Q}_\varepsilon\odot\nabla{Q}_\varepsilon-\mathcal{F}({Q}_\varepsilon)I_3\big)
: (\mathcal{A}(\rho_\varepsilon)\otimes\nabla\phi) \,dxdt
\non\\
&&-\int_0^T\psi\int_{U}\phi\big(L\nabla{Q}_\varepsilon\odot\nabla{Q}_\varepsilon-\mathcal{F}({Q}_\varepsilon)I_3\big)
: \nabla\mathcal{A}(\rho_\varepsilon) \,dxdt \non\\
&\doteq&I_{9a}+I_{9b}. \label{hard term 1} \\
J_{9}&=&-\int_0^T\psi\int_{U}\big(L\nabla{Q}\odot\nabla{Q}-\mathcal{F}(Q)I_3\big)
: \mathcal{A}(\rho)\otimes\nabla\phi \,dxdt
\non\\
&&-\int_0^T\psi\int_{U}\phi\big(L\nabla{Q}\odot\nabla{Q}-\mathcal{F}({Q})I_3\big)
: \nabla\mathcal{A}(\rho) \,dxdt \non\\
&\doteq&J_{9a}+J_{9b}. \label{limit term 1} \eea Due to \eqref{var
conv 3} and \eqref{var conv 8}, the convergence of $I_{9a}$ to
$J_{9a}$ is straightforward. While for $I_{9b}$ and $J_{9b}$, by the
property of the singular integral operator $\mathcal{A}$, it holds
\bea
&&I_{9b}-J_{9b}\non\\
&=&-L\int_0^T\psi\int_{U}\phi(\nabla{Q}_\varepsilon-\nabla{Q})\odot\nabla{Q}_\varepsilon
:\nabla\mathcal{A}(\rho_\varepsilon)\,dxdt \non\\
&&-L\int_0^T\psi\int_{U}\phi\nabla{Q}\odot(\nabla{Q}_\varepsilon-\nabla{Q})
: \nabla\mathcal{A}(\rho_\varepsilon)\,dxdt    \non\\
&&-L\int_0^T\psi\int_{U}\phi\nabla{Q}\odot\nabla{Q} :
\nabla\big(\mathcal{A}(\rho_\varepsilon)-\mathcal{A}(\rho)\big)\,dxdt
\non\\
&&+\int_0^T\psi\int_{U}\phi\big(\mathcal{F}(Q_\varepsilon)-\mathcal{F}(Q)\big)\rho_\varepsilon\,dxdt
+\int_0^T\psi\int_{U}\phi\mathcal{F}(Q)(\rho_\varepsilon-\rho)\,dxdt
\non\\
&\doteq&K_{9ba}+K_{9bb}+K_{9bc}+K_{9bd}+K_{9be}. \eea Using
\eqref{L1 est 2}, \eqref{var conv 3} and \eqref{Mikhlin Multiplier
theorem}, we find $K_{9ba} \rightarrow 0$, $K_{9bb} \rightarrow 0$.
By \eqref{L1 est 7}, \eqref{L1 est 8}, \eqref{var conv 3},
\eqref{var conv 4} and Lemma \ref{lemma on density estimate without
varepsilon}, we know $K_{9bd} \rightarrow 0$, $K_{9be} \rightarrow
0$. As for $K_{9bc}$, we deduce from \eqref{var conv 9} that for
a.e. fixed $t \in [0, t]$, it holds
\[ \psi(t)\int_{U}\phi(x)\nabla{Q}\odot\nabla{Q} : \big(\nabla\mathcal{A}(\rho_\varepsilon)-\nabla\mathcal{A}(\rho) \big)\,dx
\rightarrow 0, \ \ \mbox{as } \varepsilon \rightarrow 0. \]
Meanwhile, since $\beta > 4$, using Holder's inequality, we obtain
from \eqref{Mikhlin Multiplier theorem} and Lemma \ref{lemma on
density estimate without varepsilon} that $\forall \varepsilon > 0,
\forall t \in [0, T]$, \bea &&\Big|
\psi(t)\int_{U}\phi(x)\nabla{Q}\odot\nabla{Q} :
\big(\nabla\mathcal{A}(\rho_\varepsilon)-\nabla\mathcal{A}(\rho)
\big)\,dx \Big| \non\\
&\leq&C\|\nabla{Q}\|_{L^{\frac{10}{3}}(U)}^2\|\nabla\mathcal{A}(\rho_\varepsilon)-\nabla\mathcal{A}(\rho)\|_{L^{\frac52}(U)}  \non\\
&\leq& C(E_\delta(\rho_0, q_0, Q_0), a,b,c,\beta, L, U,
T)\|\nabla{Q}\|_{L^{\frac{10}{3}}(U)}^2, \non \eea with the right
hand side term being integrable on $(0, T)$ due to \eqref{L1 est
10}. Hence we conclude that $K_{9bc} \rightarrow 0$ after applying
Lebesgue's convergence theorem. In all, we prove
\[ I_9 \rightarrow J_9 \ \ \mbox{as } \varepsilon \rightarrow 0. \]

\noindent For $I_{10}$, we have \bea
I_{10}&=&L\int_0^T\psi\int_{U}\big(Q_\varepsilon\mathcal{H}(Q_\varepsilon)-\mathcal{H}(Q_\varepsilon){Q}_\varepsilon\big)
: \mathcal{A}(\rho_\varepsilon)\otimes\nabla\phi \,dxdt
\non\\
&&+L\int_0^T\psi\int_{U}\big(Q_\varepsilon\mathcal{H}(Q_\varepsilon)-\mathcal{H}(Q_\varepsilon){Q}_\varepsilon\big)
: \phi\nabla\mathcal{A}(\rho_\varepsilon) \,dxdt  \non\\
&\doteq&I_{10a}+I_{10b}. \label{hard term 2} \eea Notice that
$Q_\varepsilon=Q_\varepsilon^{T}$, hence
$Q_\varepsilon\mathcal{H}(Q_\varepsilon)-\mathcal{H}(Q_\varepsilon)Q_\varepsilon$
is skew-symmetric. And it is observed that $\nabla\mathcal{A}$ is
symmetric. Therefore, we conclude \be I_{10b}=0. \label{higher order
term } \ee
%%%%%%%%%%%%%%%%%%%%%%%%%%%%%%%%%%%%%%%%%%%%%%%%%%%%%%%%%%%%%%%%%%%%%%%%%
 \br We want to point out that
the special property of $Q$-tensor is of great importance here, for
otherwise we are not able to control the higher order terms in
$I_{10b}$. \er We proceed to show the convergence of
$I_{10}$ to $J_{10}$. \bea I_{10}-J_{10}&=&I_{10a}-J_{10} \non\\
&=&L\int_0^T\psi\int_{U}\big(Q_\varepsilon\Delta{Q}_\varepsilon-\Delta{Q}_\varepsilon{Q}_\varepsilon\big)
:
\big(\mathcal{A}(\rho_\varepsilon)-\mathcal{A}(\rho)\big)\otimes\nabla\phi
\,dxdt \non\\
&&+L\int_0^T\psi\int_{U}\big((Q_\varepsilon-Q)\Delta{Q}_\varepsilon-\Delta{Q}_\varepsilon({Q}_\varepsilon-Q)\big)
: \mathcal{A}(\rho)\otimes\nabla\phi \,dxdt \non\\
&&+L\int_0^T\psi\int_{U}\big(Q_\varepsilon(\Delta{Q}_\varepsilon-\Delta{Q})-(\Delta{Q}_\varepsilon-\Delta{Q}){Q}_\varepsilon\big)
: \mathcal{A}(\rho)\otimes\nabla\phi \,dxdt \non\\
&\doteq&K_{10aa}+K_{10ab}+K_{10ac}. \non \eea By \eqref{L1 est 7},
\eqref{L1 est 8}, \eqref{L1 est 9}, \eqref{var conv 3} and
\eqref{var conv 8}, it is easy to see that
\[ K_{10aa} \rightarrow 0, \ \ K_{10ab} \rightarrow 0, \ \ K_{10ac}\rightarrow 0, \ \ \mbox{as } \ \varepsilon \rightarrow 0, \]
hence
\[ I_{10}\rightarrow J_{10}, \ \ \mbox{as } \ \varepsilon \rightarrow 0. \]
Summing up all the above convergence results, we finish the proof of
Lemma \ref{lemma on viscous flux}.

\end{proof}
%%%%%%%%%%%%%%%%%%%%%%%%%%%%%%%%%%%%%%%%%%%%%%%%%%%%%%%%%%%%%%%%%%%%%%%%%%%%%%%%%%%%%%%%%%%%
\subsection{Strong convergence of density}
In this subsection we shall show that
\[p=\rho^\gamma+\delta\rho^\beta, \]
and consequently the strong convergence of $\rho_\varepsilon$ in
$L^1((0, T)\times U)$. By Lemma \ref{lemma on 0 extension}, we can
take the standard mollifier $\vartheta_m=\vartheta_m(x)$ to equation
\eqref{density equ viscosity}, such that \be
\partial_tS_m(\rho)+\mbox{div}\,(S_m(\rho)u)=r_m, \ \ \mbox{on } (0, T)\times\mathbb{R}^3,
\label{density equ smoothing} \ee with $S_m(\rho)=\vartheta \ast
\rho$ and $r_m \rightarrow 0$ in $L^1((0, T)\times U)$ (c.f.
\cite{L98-1}). Then for any $g$ satisfying \eqref{property of
renormailized}, we can multiply \eqref{density equ smoothing} with
$g'(S_m(\rho))$ and pass to the limit as $m \rightarrow \infty$.
Then we may argue that (\cite{DL89}) $(\rho, u)$ solve
\eqref{density equ viscosity} in the sense of renormalized
solutions, namely, \eqref{renormalized form of density equ} holds in
$\mathcal{D}'((0, T)\times U)$. Instead of the strong restrictions
on $g$ in \eqref{property of renormailized}, one can use the
Lebesgue convergence theorem to relax the assumptions in Definition
\ref{def of finite energy weak solution} to any function $b \in
C^1(0, \infty) \cap C[0, \infty)$ with
\[ |g'(z)z| \leq C(z^\theta+z^{\frac{\gamma}{2}}), \ \
\forall \, z >0 \ \mbox{and some } 0< \theta<\frac{\gamma}{2}.   \]
Hence we may choose $g(z)=z\ln(z)$ and integrate \eqref{renormalized
form of density equ} to obtain
%%%%%%%%%%%%%%%%%%%%%%%%%%%%%%%%%%%%%%%%%%%%%%%%%%%%%%%%%%%%%%%%%%%%%%%%
 \be  \int_0^T\int_{U}\rho\mbox{div}\,u\,dxdt=\int_{U}\rho_0\ln(\rho_0)dx-\int_{U}\rho(T)\ln(\rho(T))dx.
 \label{density inequality-1 renormalized} \ee
Meanwhile, using Lemma \ref{lemma on density regularity} and the
convexity of $g(z)=z\ln(z)$, we know \[
\partial_tg(\rho_\varepsilon)+\mbox{div}(g(\rho_\varepsilon)u_\varepsilon)
+\rho_\varepsilon\mbox{div}{u}_\varepsilon-\varepsilon\Delta{g}(\rho_\varepsilon)
\leq 0, \] which leads to
%%%%%%%%%%%%%%%%%%%%%%%%%%%%%%%%%%%%%%%%%%%%%%%%%%%%%%%%%%%%%%%%%%%%%%%%
\be
\int_0^T\int_{U}\rho_\varepsilon\mbox{div}\,u_\varepsilon\,dxdt=\int_{U}\rho_0\ln(\rho_0)dx-\int_{U}\rho_\varepsilon(T)\ln(\rho_\varepsilon(T))dx.
\label{density inequality-2 renormalized} \ee Taking two
nondecreasing sequences $\phi_n \in \mathcal{D}(0, T)$, $\phi_n \in
\mathcal{D}(U)$ of nonnegative functions with $\psi_n \rightarrow 1,
\phi_n \rightarrow 1$ as $n \rightarrow \infty$. By Lemma \ref{lemma
on viscous flux}, \eqref{density inequality-1 renormalized} and
\eqref{density inequality-2 renormalized}, one can apply standard
arguments to show that
\[ \displaystyle\lim\sup_{\varepsilon \rightarrow 0^{+}}\int_0^{T}\psi_n\int_{U}\phi_n\rho_\varepsilon^\gamma
+\delta\rho_\varepsilon^\beta)\rho_\epsilon\,dxdt \leq
\int_0^T\int_{U}P\rho\,dxdt, \ \ \text{ for all }\, n=1,2,\cdots
\]
Notice that $P(z)=z^\gamma+\delta{z}^\beta$ is monotone, by Minty's
trick, we have
\[ \int_0^{T}\psi_m(t)\int_{U}\phi_m(x)\big(P(\rho_\varepsilon)-P(v)\big)(\rho_\epsilon-v)dxdt \geq 0. \]
Consequently, taking $n\rightarrow \infty$, we obtain after
rearrangement that for any $v=\rho+\kappa\phi$, $\phi \in
\mathcal{D}(U)$, it holds
\[ \int_0^{T}\int_{U}\big(p-P(v)\big)(\rho-v)dxdt \geq 0.  \]
Let $\kappa \rightarrow 0$, we come to the conclusion
\[ p=\rho^\gamma+\delta\rho^\beta. \]
In all, we may summarize the above results in the following
proposition.
\begin{proposition} \label{proposition on level 2 existence}
Suppose $\beta > \max\{\frac{6\gamma}{2\gamma-3}, \gamma, 4\}$. Then
for any given $T>0$ and $\delta>0$, there exists a finite energy
weak solution $(\rho, u, Q)$ to the problem \bea
\rho_t+\mbox{div}(\rho{u})&=&0, \label{density equ viscosity-2}\\
(\rho{u})_t+\nabla\cdot(\rho{u}\otimes{u})+\nabla({\rho^\gamma+\delta\rho^\beta})&=&\mathcal{L}u-\nabla\cdot\big(L\nabla{Q}\odot\nabla{Q}
-\mathcal{F}(Q)I_3\big)\non\\
&&+L\nabla\cdot\big(Q\mathcal{H}(Q)-\mathcal{H}(Q)Q \big), \label{navier-stokes viscosity-2}\\
Q_t+u\cdot\nabla{Q}-\Omega{Q}+Q\Omega&=&\Gamma\mathcal{H}(Q),
\label{order parameter viscosity-2}
 \eea
with initial and boundary conditions \eqref{IC 1 Gakerkin}-\eqref{BC
2 Galerkin}. Furthermore, $\rho \in L^{\beta+1}((0, T)\times U)$ and
the equation \eqref{density equ viscosity-2} is satisfied in the
sense of renormalized solutions on $\mathcal{D}'((0,
T)\times\mathbb{R}^3)$ provided $\rho, u$ are extended to be zero on
$\mathbb{R}^3 \setminus U$. In addition, the following estimates are
valid: \bea \displaystyle\sup_{t\in [0,
T]}\|\rho(t)\|_{L^\gamma(U)}^\gamma &\leq& C(E_\delta(\rho_0, q_0, Q_0), \gamma),  \label{L2 est 1}\\
\delta\displaystyle\sup_{t\in [0, T]}\|\rho(t)\|_{L^\beta(U)}^\beta
&\leq& C(E_\delta(\rho_0, q_0, Q_0), \beta),  \label{L2 est 2}\\
\displaystyle\sup_{t\in [0,
T]}\big\|\sqrt{\rho}(t)u(t)\big\|_{L^2(U)}^2 &\leq& 2E_\delta(\rho_0, q_0, Q_0),  \label{L2 est 3} \\
\|u\|_{L^2(0, T; H^1_0(U))} &\leq& C(E_\delta(\rho_0, q_0, Q_0), \lambda, \nu), \label{L2 est 4}  \\
%\|\rho\|_{L^{\beta+1}((0, T)\times U)} \leq C(\varepsilon, \delta,
%\rho_0, q_0, Q_0, a, b, c, U),
%\label{L2 est 5}  \\
\|Q\|_{L^{10}((0, T)\times U)} &\leq& C(E_\delta(\rho_0, q_0, Q_0), a, b, c, L, \Gamma, U, T), \label{L2 est 6} \\
\|Q\|_{L^{\infty}(0, T; H^1(U))} &\leq&
\frac{2}{L}E_\delta(\rho_0,q_0,Q_0)
\label{L2 est 7} \\
\|\nabla{Q}\|_{L^\frac{10}{3}((0,T)\times U)} &\leq& C(E_\delta(\rho_0, q_0, Q_0), a, b, c, L, \Gamma, U, T)      \\
 \|Q\|_{L^2(0, T; H^2(U))} &\leq&  C(E_\delta(\rho_0, q_0,
Q_0), a, b, c, L, \Gamma, U, T). \label{L2 est 8}
 \eea

\end{proposition}
\br The initial conditions \eqref{IC 1 Gakerkin}-\eqref{IC 2
Galerkin} are satisfied in the weak sense, since we infer from
\eqref{var conv 4} and \eqref{var conv 6} that
\[  \rho_\varepsilon \rightarrow
\rho \ \mbox{in } C(0, T; L_{weak}^\beta(U)), \ \
\rho_\varepsilon{u}_\varepsilon \rightarrow \rho{u} \ \mbox{in }
C\big([0, T]; L_{weak}^{\frac{2\gamma}{\gamma+1}}(U)\big).  \]

\er

%%%%%%%%%%%%%%%%%%%%%%%%%%%%%%%%%%%%%%%%%%%%%%%%%%%%%%%%%%%%%%%%%%%%%%%%%%%%%%%%%%%%%%%%%%%%%%%%%
\section{Vanishing artificial pressure}\setcounter{equation}{0}
In this section, we denote by $(\rho_\delta, u_\delta, Q_\delta)$
the corresponding approximate solutions constructed in Proposition
\ref{proposition on level 2 existence}. We are going to finish the
third level approximation, namely, we shall provide the convergence
of solutions of $(\rho_\delta, u_\delta, Q_\delta)$ to the solution
of the original problem \eqref{density equ}-\eqref{order parameter}
as $\delta$ goes to $0$.

To begin with, we relax the conditions on the general initial data
$(\rho_0, u_0, Q_0)$. It is easy to find a sequence $\rho_{\delta}
\in C_0^3(\bar{U})$ with the property \[ 0 \leq \rho_{\delta}(x)
\leq \frac12\delta^{-\frac{1}{\beta}}, \ \ \mbox{and } \
\|\rho_{\delta}-\rho_0\|_{L^2(U)} < \delta.
\]Taking $\rho_{0, \delta}=\rho_\delta+\delta$, due to \eqref{IC 1 Gakerkin}, then we have
\be 0 < \delta \leq \rho_{0, \delta} \leq \delta^{-\frac{1}{\beta}},
\ \ \frac{\partial\rho_{0, \delta}}{\partial\vec{n}}=0,
\label{modified BC 1} \ee with \be \rho_{0, \delta} \rightarrow
\rho_0 \ \mbox{ in} \ L^\gamma(U) \ \ \mbox{as} \ \delta \rightarrow
0. \label{delta cond1} \ee Set \bea \tilde{q}_{\delta}(x)=\left\{
\begin{array}{l}
q(x)\sqrt{\frac{\rho_{0, \delta}}{\rho_0}}, \ \ \mbox{if } \rho_0(x)>0,\\
0, \ \ \ \ \ \ \ \ \ \ \ \ \ \mbox{if } \rho_0(x)=0.
\end{array}\right.
\eea Then it follows from \eqref{compatibility condition} that
$\frac{|\tilde{q}_\delta|^2}{\rho_{0, \delta}}$ is uniformly bounded
in $L^1(U)$. At the same time, it is easy to find $h_\delta \in
C^2(\bar{U})$ such that
\[ \Big\|\frac{\tilde{q}_\delta}{\sqrt{\rho_{0, \delta}}}-h_\delta \Big\|_{L^2(U)} < \delta. \]
Consequently, we choose $q_\delta=h_\delta\sqrt{\rho_{0, \delta}}$
and one can readily check that \be \frac{|q_\delta|^2}{\rho_{0,
\delta}} \ \mbox{are uniformly bounded in } L^1(U), \ \
\label{modified BC 2} \ee and \be q_\delta \rightarrow q \ \mbox{in
} L^1(U) \ \mbox{ as } \delta \rightarrow 0. \label{delta cond2} \ee
%%%%%%%%%%%%%%%%%%%%%%%%%%%%%%%%%%%%%%%%%%%%%%%%%%%%%%%%%%%%%%%%%%%%%%%%%%%%%%%%%%%%%%%%%
In what follows, we shall deal with the sequence of approximate
solutions $(\rho_\delta, u_\delta, Q_\delta)$ to the problem
\eqref{density equ viscosity-2}-\eqref{order parameter viscosity-2}
with the initial data $(\rho_\delta, q_\delta, Q_0)$. \br We want to
point out that due to the above modifications, the estimates
\eqref{L2 est 1}-\eqref{L2 est 8} are independent of $\delta$
because the constant $E_\delta(\rho_{0, \delta}, q_{0,\delta}, Q_0)$
defined in \eqref{delta energy bound} is independent of $\delta$.
\er

%\subsection{Local pressure estimates}
Now we shall develop some pressure estimates independent of $\delta
> 0$.
%which implies the local boundedness of the pressure in a
%reflexive $L^r$ space (with $r>1$). Different from all previous a
%priori estimates
Notice that the continuity equation \eqref{density equ viscosity-2}
is satisfied in the sense of renormalized solutions in
$\mathcal{D}'((0, T)\times \mathbb{R}^3)$, hence we may apply the
standard mollifying operator to both sides of \eqref{renormalized
form of density equ} and get \be
\partial_tS_m[g(\rho)]+\mbox{div}\big(S_m[g(\rho)u]\big)+S_m\big[(g'(\rho)\rho-g(\rho))\mbox{div}u \big]
=r_m, \ee with
\[ r_m \rightarrow 0 \ \ \mbox{in } L^2(0, T; L^2(\mathbb{R}^3)) \ \mbox{as } m \rightarrow \infty. \]
Using the operator $\mathcal{B}$ introduced in the proof of Lemma
\ref{lemma on density estimate without varepsilon}, we take the test
function to \eqref{navier-stokes viscosity-2} to be
\[ \phi_i(t, x)=\psi(t)\mathcal{B}_i\Big\{S_m[g(\rho_\delta)]-\frac{1}{|U|}\int_{U}S_m[g(\rho_\delta)]dx   \Big\},
\ \ i=1,2,3, \ \psi \in \mathcal{D}(0, T). \] Next, we can
approximate the function $g(z)$ by a sequence of function
$\{z^\theta\chi_n(z)\}$, where each $\chi_n(z)$ being a cutoff
function such that $\chi_n(z)=1$ on $[0, n]$ and $\chi_n(z)=0$ on
$z>2n$. Then using all the estimates \eqref{L2 est 1}-\eqref{L2 est
8}, we have
\begin{lemma} \label{lemma on higher density integrability}
For $\gamma > \frac{3}{2}$, there exists a constant $\theta$ that
only depends on $\gamma$, such that
\[  \int_0^T\int_{U}\left(\rho_\delta^{\gamma+\theta}+\delta\rho_\delta^{\beta+\theta}\right)\,dxdt \leq
C(\rho_0, q_0, Q_0, a, b, c, \lambda,\nu,\gamma, \beta, \Gamma, L,
U, T),
\] provided $0<\theta < \min\big\{1, \frac{\gamma}{3},
\frac{2\gamma}{3}-1 \big\}$
\end{lemma}
\begin{proof}
Since the technique is quite similar to Lemma \ref{lemma on density
estimate without varepsilon}, we shall skip the details of proof and
leave it to interested readers. It is noted that the right hand side
bound is independent of $\delta$.
\end{proof}
%%%%%%%%%%%%%%%%%%%%%%%%%%%%%%%%%%%%%%%%%%%%%%%%%%%%%%%%%%%%%%%%%%%%%%%%%%%%%%%%%%%%%%%%
\subsection{The limit passage and the effective viscous flux}
We conclude from the uniform estimates \eqref{L2 est 1}-\eqref{L2
est 8} in Proposition \ref{proposition on level 2 existence} and
Lemma \ref{lemma on higher density integrability} that \bea
&&\rho_\delta \rightarrow \rho \ \mbox{ in} \ C\big([0, T];
L^\gamma_{weak}(U)\big), \\
&&\rho_\delta \rightarrow \overline{\rho^\gamma} \ \mbox{ weakly in} \ L^{\frac{\gamma+\theta}{\gamma}}((0, T)\times U),     \\
&&u_\delta \rightarrow u \ \mbox{ weakly in} \ L^2(0, T;
H_0^1(U)), \\
&&\rho_\delta{u}_\delta \rightarrow \rho{u} \ \mbox{ in} \ C\big([0,
T];
L^{\frac{2\gamma}{\gamma+1}}_{weak}(U)\big), \\
&& Q_\delta \rightarrow Q \ \mbox{ weakly in} \ L^2(0, T; H^2(U)),\\
&& Q_\delta \rightarrow Q \ \mbox{ strongly in} \ L^2(0, T; H^1(U)),
\eea which infers \bea \rho_\delta{u}_\delta \otimes {u}_\delta
\rightarrow \rho{u}\otimes{u} \ \mbox{ in} \ \mathcal{D}'((0,
T)\times U), \eea and \bea
&&\nabla{Q_\delta}\odot\nabla{Q_\delta}-\mathcal{F}(Q_\delta)I_3-L\big(
Q_\delta\mathcal{H}(Q_\delta)-\mathcal{H}(Q_\delta)Q_\delta \big)
\non\\
&&\rightarrow \nabla{Q}\odot\nabla{Q}-\mathcal{F}(Q)I_3-L\big(
Q\mathcal{H}(Q)-\mathcal{H}(Q)Q \big) \ \  \mbox{in } \ L^1((0, T)
\times U). \eea Further, Lemma \ref{lemma on higher density
integrability} implies that \be \delta\rho_\delta^\beta \rightarrow
0 \ \mbox{ in} \ L^1((0, T)\times U). \ee Therefore, the limit
$(\rho, u, Q)$ satisfies \bea \rho_t+\mbox{div}(\rho{u})&=&0, \ \
\mbox{in } \ \mathcal{D}'\big((0, T)\times \mathbb{R}^3 \big),
\label{density equ delta}  \\
(\rho{u})_t+\nabla\cdot(\rho{u}\otimes{u})+\nabla\overline{\rho^\gamma}&=&\mathcal{L}u-\nabla\cdot\big(L\nabla{Q}\odot\nabla{Q}
-\mathcal{F}(Q)I_3\big)\non\\
&&+L\nabla\cdot\big(Q\mathcal{H}(Q)-\mathcal{H}(Q)Q \big), \label{navier-stokes delta}\\
Q_t+u\cdot\nabla{Q}-\Omega{Q}+Q\Omega&=&\Gamma\mathcal{H}(Q),
\label{order parameter delta} \eea in $\mathcal{D}'\big((0, T)\times
U \big)$. And the initial data \eqref{IC} is satisfied due to
\eqref{delta cond1} and \eqref{delta cond2}.

In what follows, our ultimate goal is to show
$\overline{\rho^\gamma}=\rho^\gamma$, or equivalently, the strong
convergence of $\rho_\delta$ in $L^1$. Consider a family of cut-off
functions by $T_k(z)=kT(\frac{z}{k})$ for $z\in \mathbb{R}$,
$k=1,2,3\cdots$ and $T \in C^\infty(R)$ is chosen to be
\[ T(z)=z \ \mbox{for} \ z \leq 1, \ \ T(z)=2 \ \mbox{for} \ z \geq 3, \ \ T \mbox{ is } concave.  \]
Since $(\rho_\delta, u_\delta)$ is a normalized solution to
\eqref{density equ delta}, it holds
%%%%%%%%%%%%%%%%%%%%%%%%%%%%%%%%%%%%%%%%%%%%%%%%%%%%%%%%%%%%%%%%%%%%%%%%%%%%%%%%%%%%%%%%%%%%%%
 \be T_k(\rho_\delta)_t+\mbox{div}\big(T_k(\rho_\delta){u}_\delta \big)+\big(T_k'(\rho_\delta)-T_k(\rho_\delta)
 \big)\mbox{div}u_\delta=0,
 \ \ \mbox{in } \  \mathcal{D}'\big((0, T)\times \mathbb{R}^3 \big),   \ee
from which we get after passing to limit for $\delta \rightarrow 0$
that
 \be \overline{T_k(\rho)}_t+\mbox{div}\big(\overline{T_k(\rho)}{u} \big)
 +\overline{\big(T_k'(\rho)-T_k(\rho)
 \big)\mbox{div}u}=0,
 \ \ \mbox{in } \  \mathcal{D}'\big((0, T)\times \mathbb{R}^3 \big).
 \label{renormalized equ by concave fun}  \ee
Here \be \big(T_k'(\rho_\delta)-T_k(\rho_\delta)
 \big)\mbox{div}u_\delta \rightarrow \overline{\big(T_k'(\rho)-T_k(\rho)
 \big)\mbox{div}u} \ \ \ \mbox{weakly in } \ L^2((0, T)\times U),   \ee
and \be T_k(\rho_\delta) \rightarrow \overline{T_k(\rho)} \ \
\mbox{in } \ C\big(0, T; L^p_{weak}(U)  \big), \ \ \forall \, 1 \leq
p < \infty. \ee By similar arguments as in the proof of Lemma
\ref{lemma on viscous flux}, we have the following auxiliary result:
\begin{lemma} \label{lemma on delta viscous flux}
Suppose $(\rho_\delta, u_\delta)$ is a sequence of approximate
solutions constructed in Proposition \ref{proposition on level 2
existence}, then for any $\psi \in \mathcal{D}(0, T)$, $\phi \in
\mathcal{D}(U)$, it holds \bea &&\displaystyle\lim_{\delta
\rightarrow
0^{+}}\int_0^T\psi(t)\int_{U}\phi(x)\big(\rho_{\delta}^\gamma
-(\lambda+2\nu)\mbox{div}\,u_{\delta}
\big)T_k(\rho_\delta)\,dxdt \non\\
&=&\int_0^T\psi(t)\int_{U}\phi(x)\big(\overline{\rho^\gamma}-(\lambda+2\nu)\mbox{div}\,u\big)\overline{T_k(\rho)}\,dxdt
\label{delta limit of viscous flux} \eea
\end{lemma}
%%%%%%%%%%%%%%%%%%%%%%%%%%%%%%%%%%%%%%%%%%%%%%%%%%%%%%%%%%%%%%%%%%%%%%%%%%%%%%%%%%%%%%%%%%%%%%%%%%%%%%%%%
\subsection{The renormalized solutions and strong convergence of density}
%First we broaden our analysis to describe possible density
%oscillations. A suitable tool to do this is a defect measure
%"$T_k(\rho)-\overline{T_k(\rho)}$", the time evolution of which is
%governed by \eqref{renormalized equ by concave fun} and
%\eqref{renormalized form of density equ}, provided that we can show
%that the limit density $\rho$ satisfies \eqref{renormalized form of
%density equ}.

As in \cite{FNP01}, we introduce a quantity namely oscillations
defect measure. To consider the weak convergence of the sequence
$\{\rho_\delta \}_{\delta>0}$ in $L^1((0, T)\times U)$, we define
\be \mathbf{osc}_{\gamma+1}[\rho_\delta-\rho] \equiv
\displaystyle\sup_{k\geq
1}\left(\displaystyle\lim\sup_{\delta\rightarrow
0}\int_0^T\int_{U}\big|T_k(\rho_\delta)-T_k(\rho)
\big|^{\gamma+1}\,dxdt \right), \ee where $T_k$ are the cut-off
functions defined above. First by virtue of Lemma \ref{lemma on
delta viscous flux}, we claim the following result concerning the
oscillation defect measure.

\bl \label{lemma on OSC} There exists a constant $C$ independent of
$k$, such that
\[ \mathbf{osc}_{\gamma+1}[\rho_\delta-\rho] \leq C.  \]
 \el
\begin{proof}
Notice that $z^\gamma$ is a convex function for  $\gamma > \frac32$,
we have (see Theorem 2.11 in \cite{F04})
\[ \rho^\gamma \leq \overline{\rho^\gamma}, \ \ \ z^\gamma-y^\gamma \geq (z-y)^\gamma, \
 \ \mbox{for } \ z\geq y \geq 0. \]
Meanwhile, since $T_k(z)$ is concave, we know
\[ |T_k(z)-T_k(y)| \leq |z-y|, \ \ \ T_k(\rho) \geq \overline{T_k(\rho)}, \ \ \forall k \geq 1,  \]
and henceforth
\[ |T_k(z)-T_k(y)|^{\gamma+1} \leq |z-y|^\gamma |T_k(z)-T_k(y)| \leq (z^\gamma-y^\gamma)\big(T_k(z)-T_k(y)\big).  \]
Consequently, it yields \bea
&&\displaystyle\lim\sup_{\delta\rightarrow
0}\int_0^T\int_{U}|T_k(\rho_\delta)-T_k(\rho)|^{\gamma+1}dxdt \non\\
&\leq&\displaystyle\lim_{\delta\rightarrow 0}\int_0^T\int_{U}
(\rho_\delta^\gamma-\rho^\gamma)\big(T_k(\rho_\delta)-T_k(\rho)\big)\,dxdt
+\int_0^T\int_{U} (\overline{\rho^\gamma}-\rho^\gamma)\big(T_k(\rho)-\overline{T_k(\rho)} \big)\,dxdt \non\\
&=&\displaystyle\lim_{\delta\rightarrow 0}\int_0^T\int_{U}
\rho_\delta^\gamma
T_k(\rho_\delta)-\overline{\rho^\gamma}\,\overline{T_k(\rho)}\,dxdt
\non\\
&=& \nu\displaystyle\lim_{\delta\rightarrow
0}\int_0^T\int_{U}\mbox{div}u_\delta T_k(\rho_\delta)-\mbox{div}u
T_k(\rho)\,dxdt  \non\\
&\leq& \nu\displaystyle\lim_{\delta\rightarrow 0}\int_0^T\int_{U}
\big(T_k(\rho_\delta)-T_k(\rho)+T_k(\rho)-\overline{T_k(\rho)}\,\big) \mbox{div}u_\delta \,dxdt  \non\\
&\leq&C \displaystyle\sup_{\delta>0}\|\mbox{div}u_\delta\|_{L^2((0,
T)\times U)}\displaystyle\lim\sup_{\delta\rightarrow
0}\|T_k(\rho_\delta)-T_k(\rho)\|_{L^{\gamma+1}((0, T)\times U)},
 \eea
where we applied Lemma \ref{lemma on delta viscous flux} in the
third step.

\end{proof}
Based on the uniform bound for oscillation defect measure shown in
Lemma \ref{lemma on OSC}, we can apply the same argument in
\cite{FNP01} to show that the limit functions $(\rho, u)$ satisfy
\eqref{density equ delta} in the sense of renormalized solutions.
\bl \label{lemma on revisited renor sol} The limit functions $(\rho,
u)$ satisfy equation \eqref{density equ delta} in the sense of
renormalized solutions, namely, \be
g(\rho)_t+\mbox{div}(g(\rho)u)+(g'(\rho)\rho-g(\rho))\mbox{div}u=0,
\ee holds in $\mathcal{D}\big((0, T)\times \mathbb{R}^3 \big)$ for
any $g$ satisfying \eqref{property of renormailized}.

\el

Finally, we shall discuss the propagation of oscillations, whose
amplitude in the sequence $\{\rho_\delta\}_{\delta>0}$ is measured
by the following quantity
\[ \mathbf{dft}[\rho_\delta\rightarrow \rho](t) \equiv \int_{U}\Big(\overline{\rho\ln(\rho)}-\rho\ln(\rho) \Big)(t, x)dx
, \ \ \ t \in [0, T]. \] To this end, we introduce the auxiliary
functions
\[ L_k(\rho) = \rho\int_1^\rho\frac{T_k(z)}{z^2}dz, \]
where $T_k$ are cutoff functions defined above. Now the equation \[
\partial_tL_k(\rho_\delta)+\mbox{div}\big(L_k(\rho_\delta)u_\delta
\big) +T_k(\rho_\delta)\mbox{div}\,u_\delta = 0
   \] holds in $\mathcal{D}'\big((0, T)\times
\mathbb{R}^3\big)$. Letting $\delta \rightarrow 0$ we obtain \be
\partial_t\overline{L_k(\rho)}+\mbox{div}\big(\overline{L_k(\rho)}u
\big)+\overline{T_k(\rho)\mbox{div}\,u} = 0.
 \label{den equ Lk limit} \ee Here $L_k(\rho) \in C([0, T]; L^1(U))$
and  \[ L_k(\rho_\delta) \rightarrow \overline{L_k(\rho)} \,\,
\mbox{in} \ C\big(0, T; L_{weak}^\gamma(U) \big), \ \
T_k(\rho_\delta)\mbox{div}\,u_\delta \rightarrow
\overline{T_k(\rho)\mbox{div}\,u} \ \mbox{weakly in} \ L^2((0,
T)\times U).
\] By Lemma \ref{lemma on revisited renor sol}, the limits $(\rho, u)$
satisfy \be
\partial_tL_k(\rho)+\mbox{div}(L_k(\rho)u)+T_k(\rho)\mbox{div}u = 0
\ \ \mbox{ in } \ \mathcal{D}'\big((0, T)\times \mathbb{R}^3\big).
 \label{den equ Lk}\ee
Taking the difference between \eqref{den equ Lk limit} and
\eqref{den equ Lk}, then taking the inner product of the resultant
with a test function $\psi(t)\phi(x)$, with $\psi \in \mathcal{D}(0,
T)$ and $\phi \in \mathcal{D}(\mathbb{R}^3)$ with $\phi \equiv 1$ on
an open neighborhood of $\bar{U}$, we get after integrating from $0$
to $t$ that \bea &&\int_{U}\big(\overline{L_k(\rho)}-L_k(\rho)
\big)(t)dx\non\\
&=&\int_0^{t}\int_{U}\big(\overline{T_k(\rho)}\mbox{div}\,u-\overline{T_k(\rho)\mbox{div}\,u}\big)\,dxdt
+\int_0^t\int_{U}\big(T_k(\rho)-\overline{T_k(\rho)}
\big)\mbox{div}\,u\,dxdt. \label{diff equ of Lk} \eea Notice that
$T_k(z)$ is a convex function of $z \geq 0$, by Lemma \ref{lemma on
delta viscous flux} again, we deduce from \eqref{diff equ of Lk}
that for all $t \in [0, T]$, \bea 0 &\leq&
\int_{U}\big(\overline{L_k(\rho)}-L_k(\rho)
\big)(t)dx \non\\
%&=&\displaystyle\lim_{\delta\rightarrow 0^+}\int_0^t\int_{U}\big(
%\overline{T_k(\rho)}\mbox{div}\,u-T_k(\rho_\delta)\mbox{div}\,u_\delta
%\big)+\big(T_k(\rho)-\overline{T_k(\rho)}
%\big)\mbox{div}\,u\,dxd\tau  \non\\
&=&\displaystyle\lim_{\delta\rightarrow
0^+}\frac{1}{\nu}\int_0^t\int_{U}\big(\rho_\delta^\gamma
T_k(\rho_\delta)-\overline{\rho^\gamma}\overline{T_k(\rho)}
\big)dxd\tau+\int_0^t\int_{U}\big(T_k(\rho)-\overline{T_k(\rho)}
\big)\mbox{div}\,u\,dxd\tau  \non\\
&\leq&\int_0^t\int_{U}\big(T_k(\rho)-\overline{T_k(\rho)}
\big)\mbox{div}\,u\,dxd\tau  \non\\
&\leq&\|\mbox{div}u\|_{L^2((0, T)\times
U)}\|\overline{T_k(\rho)}-T_k(\rho)\|_{L^1((0, T)\times
U)}^{\frac{\gamma-1}{2\gamma}}\|\overline{T_k(\rho)}-T_k(\rho)\|_{L^{\gamma+1}((0,
T)\times U)}^{\frac{\gamma+1}{2\gamma}} \non\\
&\doteq&I. \label{I} \eea By \eqref{L2 est 1}, \eqref{L2 est 4},
Lemma \ref{lemma on OSC}, letting $k \rightarrow \infty$ in
\eqref{I}, we get \bea 0
&\leq& \mathbf{dft}[\rho_\delta\rightarrow \rho](t) \leq I \non\\
&\leq& C\displaystyle\lim_{k \rightarrow
\infty}\|\overline{T_k(\rho)}-T_k(\rho)\|_{L^1((0, T)\times
U)}^{\frac{\gamma-1}{2\gamma}} \non\\
&\leq& C\displaystyle\lim_{k \rightarrow
\infty}\|\overline{T_k(\rho)}-\rho\|_{L^1((0, T)\times
U)}^{\frac{\gamma-1}{2\gamma}}+C\displaystyle\lim_{k \rightarrow
\infty}\|T_k(\rho)-\rho\|_{L^1((0, T)\times
U)}^{\frac{\gamma-1}{2\gamma}} \non\\
&\leq&C\displaystyle\lim_{k \rightarrow \infty} \lim_{\delta
\rightarrow 0^+} \|T_k(\rho_\delta)-\rho_\delta\|_{L^1((0, T)\times
U)}^{\frac{\gamma-1}{2\gamma}} \non\\
&\leq&C\displaystyle\lim_{k \rightarrow
\infty}2^{\frac{\gamma-1}{2\gamma}}k^{-\frac{(\gamma-1)^2}{2\gamma}}
\lim_{\delta\rightarrow 0^+}\|\rho_\delta\|_{L^\gamma((0,T)\times
U)}^{\frac{\gamma-1}{2}} \non\\
&=&0,
 \eea
which indicates
\[ \overline{\rho\ln\rho}(t)=\rho\ln\rho(t), \ \ \ \mbox{for all } \ t \in [0, T]. \]
Hence we manage to prove the strong convergence of $\rho_\delta
\rightarrow \rho$ in $L^1((0, T)\times U)$.
%%%%%%%%%%%%%%%%%%%%%%%%%%%%%%%%%%%%%%%%%%%%%%%%%%%%%%%%%%%%%%%%%%%%%%%%%%%%%%%%%%%%%%%%%%%%%%%%%%%%%%%%%%%%%%%%%%%%%
\section{Long time dynamics}\setcounter{equation}{0}
Finally, in this section we discuss briefly the long time behavior
of any finite energy global weak solution $(\rho, u, Q)$. The main
result is as follows.
\begin{theorem} \label{theorem on ld}
Suppose $\gamma > \frac32$, for any finite weak energy solution to
the problem \eqref{density equ}-\eqref{BC}, there exists a steady
state solution $(\rho_s, 0, Q_s)$, with \be \rho_s=\frac{m_0}{|U|},
\ \ \mathcal{H}(Q_s)=0 \ \mbox{for} \ x \in U, \
Q_s|_{\partial_U}=Q_0, \label{equ for steady sol} \ee where
$m_0=\int_{U}\rho_0\,dx,$
 such that \be
\rho(t) \rightarrow \rho_s \ \ \mbox{weakly in} \ L^\gamma(U) \
\mbox{as } t \rightarrow \infty, \ee and \be
\displaystyle\lim_{t\rightarrow \infty}\mathcal{E}(t)=\mathcal{E}_s,
\ee where $\mathcal{E}_s$ is defined in \eqref{def of steady
energy}.
 Furthermore, there exists
an increasing sequence $\{t_n \}$ tending to infinity, for $t \in
[0, 1]$, it holds as $n \rightarrow \infty$ \bea &&u(t+t_n)
\rightarrow 0 \ \ \mbox{weakly in} \ L^2(0,1; H^1(U)),
\\
&&Q(t_n) \rightarrow Q_s \ \ \mbox{strongly in} \ L^2(0,1; H^1(U)) \
\mbox{and} \ \mbox{weakly in} \ L^2(0,1; H^2(U)). \eea
% In addition,
%if we further assume $Q_s$ is a global minimizer of the elastic
%energy $\mathcal{G}(Q)$, then the convergence result can be
%strengthened as follows: \be \rho(t) \rightarrow \rho_s \ \mbox{
%strongly in} \ L^\gamma(U), \ \ \displaystyle\lim_{t\rightarrow
%\infty}\int_{U}\rho(t)|u(t)|^2\,dx=0, \ \
%\displaystyle\lim_{t\rightarrow
%\infty}\mathcal{G}(Q(t))=\mathcal{G}(Q_s).
% \label{long time convergence}\ee

\end{theorem}
%%%%%%%%%%%%%%%%%%%%%%%%%%%%%%%%%%%%%%%%%%%%%%%%%%%%%%%%%%%%%%%%%%%%%%%%%%%%%%%%%%%%%%%%%%%%
\br The existence of a classical solution $Q_s$ in \eqref{equ for
steady sol} is guaranteed from elliptic PDE theory. The infimum
energy of $\mathcal{G}(Q)$ can be achieved, due to the weak lower
semi-continuity and coercivity of $\mathcal{G}(Q)$. \er

\begin{proof}
To begin with, we obtain from Theorem \ref{theorem on main result}
that \be
\mbox{ess}\displaystyle\sup_{t>0}\mathcal{E}(t)+\int_0^{\infty}\int_{U}\left(\nu|\nabla{u}|^2
+(\lambda+\nu)|\mbox{div}u|^2+\Gamma\mbox{tr}^2(\mathcal{H})\right)\,dxdt
\leq \mathcal{E}(0). \ee Consequently, we know from Corollary
\ref{remarkofQ} that
%%%%%%%%%%%%%%%%%%%%%%%%%%%%%%%%%%%%%%%%%%%%%%%%%%%%%%%%%%%%%%%%%%%%%%%%%%%%%%%%%%%%%%%%%%%
\be
\begin{split}
&\mbox{ess}\displaystyle\sup_{t>0}\big(\|\rho\|_{L^\gamma(U)}+\|\sqrt{\rho}u\|_{L^2(U)}+\|Q\|_{H^1(U)}
\big)\\&+\int_0^\infty\int_{U}\|\nabla{u}\|_{L^2(U)}^2+\mbox{tr}^2(\mathcal{H})\,dxdt
\leq C(\mathcal{E}_0, a,b,c, U).
\end{split} \label{uniform bound for BE} \ee For
the sake of convenience, we introduce the following sequences
\begin{eqnarray*} &&\rho_n(x, t)
\doteq \rho(x, t+n), \ \ u_n(x, t) \doteq u(x, t+n), \ \ Q_n(x, t)
\doteq Q(x, t+n), \\
&&\mathcal{H}_n(x, t)=L\Delta Q_n-a Q_n-cQ_n \mbox{tr}(Q_n^2),
\end{eqnarray*}
 for all integer $n$ and $t \in (0, 1)$, $x\in U$. Then it follows immediately from \eqref{uniform bound for BE} that for
any $n$, we have \bea &&\rho_n \in L^\infty(0, 1; L^\gamma(U)), \ \
\sqrt{\rho_n}u_n \in L^\infty(0, 1; L^2(U)), \ \ Q_n \in L^\infty(0,
1; H^1(U)), \label{uniform bound in LB}\\
&&\displaystyle\lim_{n\rightarrow
\infty}\int_0^1\left(\|\nabla{u}_n\|_{L^2(U)}^2+\|\mbox{tr}^2(\mathcal{H}_n)\|_{L^{1}(U)}\right)\,
dt=0. \label{zero limit sequence}  \eea Therefore, choosing a
subsequence if necessary, we know as $n \rightarrow \infty$ that
\bea &&\rho_n(x, t) \rightarrow \rho_s \ \ \ \mbox{weakly in }
L^\gamma\big((0, 1)\times
U\big), \\
&&u_n(x, t) \rightarrow 0 \ \ \ \ \mbox{weakly in } L^2\big(0,
1; H_0^1(U) \big), \\
&&Q_n(x, t) \rightarrow Q_s \ \ \mbox{weakly in } L^2\big(0, 1;
H^2(U) \big), \label{H2 weak convergence of Q} \\
&&H_n(x, t) \rightarrow 0 \ \ \ \ \mbox{weakly in } L^2\big(0, 1;
L^2(U) \big).\label{zero limit of H} \eea
%
%Furthermore, we have \begin{equation*} \int_{U}\rho_s dx\leq
%\lim_{n\to\infty}\inf \int_{U}\rho_n dx\leq C(E_0).
%\end{equation*}
%
On the other hand, it is easy to deduce from \eqref{uniform bound
for BE} and \eqref{zero limit sequence} that \be
\displaystyle\lim_{n\rightarrow \infty}\int_0^1
\left(\|\rho_n|u_n|^2\|_{L^\frac{3\gamma}{\gamma+3}(U)}+\|\rho_nu_n\|^2_{L^{\frac{6\gamma}{\gamma+6}}(U)}\right)\,dt=0.
\label{extra zero limit sequence} \ee

Since $\rho, u$ are solutions to \eqref{density equ} in the sense of
renormalized solutions, we take the test function sequence $\eta(x,
t)=\psi(t)\phi(x)$ in \eqref{density equ}, with $\phi(x) \in
\mathcal{D}(U)$, $\psi(t)\in \mathcal{D}(0, 1)$, to have
\[
\int_0^1 \Big( \int_{U}\rho_n(x, t)\phi(x)dx
\Big)\psi'(t)\,dt+\int_0^1\int_{U}\rho_n(x)u_n(x)\nabla\phi(x)
\psi(t)\,dxdt=0.
\]
Taking $n\rightarrow \infty$ and using \eqref{extra zero limit
sequence}, we get
\[ \int_0^1 \Big( \int_{U}\rho_s\phi(x)dx
\Big)\psi'(t)\,dt=0,  \] which indicates $\rho_s$ is a function
independent of $t$, and henceforth $m(\rho) \doteq \int_{U}\rho(x,
t) dx$ is a constant. On the other hand, by \eqref{zero limit
sequence}, \eqref{H2 weak convergence of Q} and \eqref{zero limit of
H}, we have \be \mathcal{H}(Q_s)=0. \label{zero steady Q} \ee Hence
if we apply the test function $\eta(x, t)$ again to equation
\eqref{order parameter}, we know that $Q_s$ is also a function
independent of $t$. Moreover, we infer from equation \eqref{order
parameter} and \eqref{uniform bound for BE} that
\[ \partial_tQ_n \in L^2\big((0, 1); L^{\frac{3}{2}}(U)\big) ,  \]
combined with \eqref{H2 weak convergence of Q},  we deduce by
Aubin-Lions compactness theorem that
 \be Q_n \rightarrow Q_s \ \ \mbox{strongly in } L^2(0, 1,
H^1(U)), \label{strong H1 convergence} \ee with $Q_s$ satisfying \be
\label{equationofH(Q)} \mathcal{H}(Q_s)=0, \ \ \ Q_s \in S_0^3, \ \
\mbox{a.e. in} \ U, \ \ \ Q_s|_{\partial{U}}=Q_0. \ee

Next, similar to arguments in previous sections, we can establish
the following higher integrability result for $\rho$ in $2D$:
%%%%%%%%%%%%%%%%%%%%%%%%%%%%%%%%%%%%%%%%%%%%%%%%%%%%%%%%%%%%%%%%%%%%%%%%%%
\bl \label{lemma on high integ in ld} For $\gamma > 1$, there exists
$\theta >0$, such that for all $n$, it holds
\[ \int_0^1\int_{U}\rho_n^{\gamma+\theta}(x, t)\,dxdt \leq C. \] \el
By Lemma \ref{lemma on high integ in ld}, we may assume \be
\rho_n^{\gamma} \rightarrow \overline{\rho^\gamma} \ \ \mbox{weakly
in} \ L^{\frac{\gamma+\theta}{\gamma}}\big((0, 1)\times U \big).
\label{weak convergence for rho s} \ee Thus, passing to the limit in
equation \eqref{navier-stokes}, using \eqref{uniform bound in LB},
\eqref{zero limit sequence}, and \eqref{extra zero limit sequence},
we obtain \bea
\nabla{\overline{\rho^\gamma}}&=&-\nabla\cdot\big(L\nabla{Q}_s\odot\nabla{Q}_s-\mathcal{F}(Q_s)I_3
\big)   \non\\
&=&-\nabla{Q}_s : \big[L\Delta{Q}_s-aQ_s+bQ_s^2-cQ_s{tr}(Q_s^2)\big]
\non\\
&=&-\nabla{Q}_s : \big[\mathcal{H}(Q_s)+\frac{b}{3}tr(Q_s^2)I_3
\big] \non\\
&=&-\nabla{Q}_s :
\mathcal{H}(Q_s)-\frac{b}{3}tr(Q_s^2)\nabla{tr(Q_s)}
 \non\\
&=&0 \ \ \ \ \mbox{in } \ \mathcal{D}'\big((0,1)\times{U}\big).
\label{limit of den distribution sense} \eea Next, following the
same argument as in \cite{FP99}, that is, using the $L^p$-version of
the celebrated div-curl lemma argument as in \cite{FP99}, we can
actually show that the convergence in \eqref{weak convergence for
rho s} is strong, and henceforth \be \rho_n \rightarrow \rho_s \ \
\mbox{strongly in} \ L^\gamma\big((0, 1)\times U \big).
\label{strong L gamma convergence}\ee Note that we already claim
that $\rho_s$ is a function independent of $t$, thus \eqref{limit of
den distribution sense}-\eqref{strong L gamma convergence} indicate
that \be \rho_s=\frac{m_0}{|U|}, \ee where we used a fact that
$m(\rho)=\int_{U}\rho\,dx$ is a constant, and
$m_0=\int_{U}\rho_0\,dx.$

On the other hand, by the basic energy law \eqref{basic energy law}
and Lemma \ref{lemma on lower bound of total energy}, we may assume
\be \mathcal{E}_\infty \doteq
\displaystyle\lim_{t\rightarrow\infty}\mathcal{E}(t)=
\displaystyle\lim_{t\rightarrow\infty}\left(\int_{U}\Big[\frac12\rho|u|^2(t)+\frac{\rho^{\gamma}(t)}{\gamma-1}\Big]\,dx
+\mathcal{G}(Q(t))\right). \ee And we define the energy for the
limit functions $(\rho_s, 0, Q_s)$ by \be \mathcal{E}_s \doteq
\int_{U}\frac{\rho_s^\gamma}{\gamma-1}\,dx+\mathcal{G}(Q_s).
\label{def of steady energy} \ee
%%%%%%%%%%%%%%%%%%%%%%%%%%%%%%%%%%%%%%%%%%%%%%%%%%%%%%%%%%%%%%%%%%%%%%%%%%%%%%%%%%%%%%%%%%%
 Using \eqref{extra zero limit sequence}, \eqref{strong H1 convergence} and \eqref{strong L gamma convergence}, we
 get
 \be
  \mathcal{E}_\infty =\displaystyle\lim_{n \rightarrow
  \infty}\int_0^1\mathcal{E}(\tau+n)\,d\tau = \displaystyle\lim_{n \rightarrow
  \infty}\int_0^1\left\{\int_{U}
  \Big[\frac12\rho_n|u_n|^2+\frac{\rho_n^\gamma}{\gamma-1}\Big]\,dx+\mathcal{G}(Q_n)\right\}\,d\tau
  = \mathcal{E}_s.
\label{limit and steady energy 1} \ee Finally, it is easy to derive
from equation \eqref{density equ} that
\[ \rho(t) \rightarrow \rho_s \ \ \mbox{weakly in } L^\gamma(U), \ \mbox{as } t \rightarrow \infty.  \]
\end{proof}

\bigskip
\noindent \textbf{Acknowledgments:} The authors would like to thank
Professors Arghir Zarnescu and Colin Denniston for their valuable discussions. 
D. Wang's research was supported in part by the National Science
Foundation under Grant DMS-0906160 and by the Office of Naval
Research under Grant N00014-07-1-0668. 
C. Yu's research was supported in part by the National Science
Foundation under Grant DMS-0906160.
Xu was partially supported by NSF grant DMS-0806703.

%%%%%%%%%%%%%%%%%%%%%%%%%%%%%%%%%%%%%%%%%%%%%%%%%%%%%%%%%%%%%%%%%%%%%%%%%%%%%%%%%%%%%%%%%%%%%%%%%%%%%%%%%%%%%%%%%%%%%%%%%%


\begin{thebibliography}{99}
\itemsep=0pt


\bibitem{BM}J. M. Ball, A. Majumdar, \emph{Nematic Liquid Crystals : from Maier-Saupe to a Continuum Theory,} Mol. Cryst. Liq. Cryst. 525 (2010) 1-11.

\bibitem{BZ} J. M. Ball, A. Zarnescu, \emph{Orientability and energy minimization in liquid crystal models},
 Arch. Ration. Mech. Anal., 202 (2011), no. 2, 493-535.

%\bibitem{BE94} A. N. Beris and B. J. Edwards, Thermodynamics of
%flowing systems with internal microstructure, Oxford engineering
%science series, vol. 36, Oxford University Press, Oxford, New York,
%1994.

\bibitem{B80} M. E. Bogovskii, \emph{Solution of some vector analysis
problems connected with operators div and grad} (in Russian), Trudy
Sem. S. L. Sobolev  {80} (1) (1980), 5-40.

%\bibitem{BD} D. Bresch, B. Desjardins, \emph{On the existence of global weak solutions to the Navier-Stokes
%equations for viscous compressible and heat conducting fluids.} J. Math. Pures Appl. (9) 87 (2007), no. 1, 57-90.

%\bibitem{CKN} L. Caffarelli, R. Kohn, L. Nirenberg, \emph{Partial regularity of suitable weak solutions of Navier-Stokes Equations}.
%Comm. Pure Appl. Math. 35 (1982), no. 6, 771-831.

\bibitem{DOY08} C. Denniston, E. Orlandini and J. M. Yeomans,
\emph{Lattice Boltzmann simulations of liquid crystals hydrodynamics}.
Phys. Rev. E., vol. 63, (5), 056702, (2008).

\bibitem{DL89}J. Diperna, P.-L. Lions, \emph{Ordinary
differential equations, transport theory and Sobolev spaces,}
Invent. math.,  {98} (1989), 511-547.

\bibitem{DWW} S. Ding, C. Wang and Y. Wen, \emph{Weak solution to compressible hydrodynamic flow of liquid crystals in 1D,}
 Discrete Contin. Dyn. Syst. Ser. B,  {15} (2011), no. 2, 357-371.

\bibitem{DLWW} S. Ding, J. Lin, C. Wang and H. Wen, \emph{Compressible hydrodynamic flow of liquid crystals in 1-D},
 Discrete Contin. Dyn. Syst,  {32} (2012), no. 2, 539-563.

\bibitem{DT06} D. Donatelli and K. Trivisa, \emph{On a multidimensional
model for the dynamic combustion of compressible reacting flow}.
Commun. Math. Phys.,  {265} (2006), 463-491.

\bibitem{DT07} D. Donatelli and K. Trivisa, \emph{On a multidimensional
model for the dynamic combustion of compressible reacting flow.}
Arch. Ration. Mech. Anal.,  {185} (2007), 379-408.

\bibitem{FNP01} E. Feireisl, A. Novotn\'{y} and H. Petzeltov\'{a}, \emph{On the
existence of globally defined weak solutions to the Navier-Stokes
equations}.  J. Math. Fluid Mech.  {3} (2001), 358-392.

\bibitem{FP99} E. Feireisl and H. Petzeltov\'{a}, \emph{Large-time
behavior of solutions to the Navier-Stokes equations of compressible
flow}.  Arch. Ration. Mech. Anal. {150} (1999), 77-96

\bibitem{F04}E. Feireisl, Dynamics of Viscous Compressible Fluids,
Oxford Lecture Series in Mathematics and its Applications, 26.
Oxford Science Publications. The Clarendon Press, Oxford University
Press, New York, 2004.

\bibitem{GA94} G. P. Galdi, An introduction to the mathematical
theory of the Navier-Stokes equations, I, Springer-Verlag, New York,
1994.

\bibitem{H95JDE} D. Hoff,  {\em Global solutions of the Navier-Stokes equations for multidimensional,
compressible flow with discontinuous initial data.}  J. Differential
Equations 120 (1995), 215-254.

\bibitem{H95} D. Hoff, \emph{Strong convergence to global solutions for
multidimensional flows of compressible, viscous fluids with
polytropic equations of state and discontinuous initial data}.  Arch.
Rational Mech. Anal.  {132} (1995), 1-14.

\bibitem{H97} D. Hoff, {\em Discontinuous solutions of the Navier-Stokes equations for multidimensional
heat-conducting flow,}  Arch. Rational Mech. Anal. 139 (1997),
303-354.

%\bibitem{HKL}
%R. Hardt, D. Kinderlehrer, F.-H. Lin, {\em Existence and partial
%regularity of static liquid crystal configurations.} Comm. Math. Phys. 105 (1986), no. 4, 547-570.

\bibitem{HW} X. Hu and D. Wang, 
\emph{Global Existence and large-time behavior of solutions to the three-dimensional
equations of compressible magnetohydrodynamic flows.}   Arch. Rational
Mech. Anal. 197 (2010), no. 1, 203-238.

\bibitem{HW2} X. Hu and D. Wang, 
\emph{Global solution to the three-dimensional incompressible flow of liquid crystals.} 
Comm. Math. Phys. 296 (2010), no. 3, 861-880. 

\bibitem{HuWu1}
X. Hu,  H. Wu,
{\em Long-time dynamics of the nonhomogeneous incompressible flow of nematic liquid crystals}. 	
arXiv:1202.4512 [math.AP].

\bibitem{HuWu2}
X. Hu,  H. Wu,
{\em Global solution to the three-dimensional compressible flow of liquid crystals}.
	arXiv:1206.2850 [math.AP]


\bibitem{HWW} T. Huang, C. Wang and Y. Wen, \emph{Strong solutions of the compressible nematic liquid crystal flow.}
Journal Diff. Equa., 252 (2012), no. 3, 2222-2265.



%\bibitem{JT}F. Jiang, Z. Tan, \emph{Global weak solution to the flow of liquid crystals system.}
% Mathematical Methods in the Applied Sciences, 32 (2009), no. 17, 2243-2266.

\bibitem{PG}P. G. De Gennes, \emph{The physics of liquid crystals.} Oxford, Clarendon Press.
1974.

%\bibitem{GP}P.G. DE GENNES, J. PROST, \emph{ The physics of liquid crystals,} 2nd edn. Oxford University Press, 1995

\bibitem{LW} X. Li and D. Wang, \emph{Global strong solution to the density-dependent incompressible flow of liquid crystals},
J. Differential Equations 252 (2012), no. 1, 745-767

\bibitem{Lin2}
F.-H.  Lin, {\em Nonlinear theory of defects in nematic liquid
crystals; phase transition and flow phenomena.}  Comm. Pure Appl.
Math. 42 (1989), no. 6, 789-814.

\bibitem{LL95} F.-H. Lin and C. Liu, \emph{Nonparabolic dissipative system
modeling the flow of liquid crystals}. Comm. Pure Appl. Math.,
 {XLVIII} (1995), 501--537.

\bibitem{LL2} F.-H. Lin and C. Liu, \emph{Partial regularity of the dynamic system modeling the flow
of liquid crystals}. Discrete Contin. Dynam. Systems. 2 (1996), no.
1, 1-22.


\bibitem{LLW} F.-H. Lin, J. Lin and C. Wang, \emph {Liquid Crystal flows in Two Dimensions.}
Arch. Ration. Mech. Anal. 197 (2010), no. 1, 297-336.

\bibitem{LW10} F.-H. Lin and C. Wang, {\em On the uniqueness of heat flow of harmonic maps and hydrodynamic
flow of nematic liquid crystals},  Chin. Ann. Math. Ser. B,
\textbf{31} (2010), no. 6, 921¨C938

%\bibitem{L98-1} P.-L. Lions, Mathematical topics in fluid mechanics.
%Vol. 1. Compressible models. Oxford Lecture Series in Mathematics
%and its Applications, 10. Oxford Science Publications. The Clarendon
%Press, Oxford University Press, New York, 1998.

\bibitem{L98} P.-L. Lions, Mathematical topics in fluid mechanics.
Vol. 2. Compressible models. Oxford Lecture Series in Mathematics
and its Applications, 10. Oxford Science Publications. The Clarendon
Press, Oxford University Press, New York, 1998.

\bibitem{LSU68} O. A. Ladyzhenskaya, N. A. Solonnikov and N. N. Uraltseva, Linear and quasilinear equations of parabolic
type, Transl. Math. Monographs, Vol. 23, American Mathematical
Society, 1968.

\bibitem{LQ} X.-G. Liu and J. Qing, \emph{Globally weak solutions to the flow of compressible liquid crystals
system}, Discrete Contin. Dyn. Syst., {33} (2) (2013),
757-788.

\bibitem{M10} A. Majumdar, \emph{Equilibrium order parameters of nematic
liquid crystals in the Landau-De Gennes theory}.  European Journal of
Applied Mathematics, {21} (2010), 181-203


\bibitem{MN79} A. Matsumura and T. Nishida,{\em The initial value problem for the equations of
motion of compressible viscous and heat-conductive fluids,} Proc.
Japan Acad. Ser. A Math. Sci. 55 (1979), 337-342.

\bibitem{MN80}
A. Matsumura and T. Nishida, \emph{The initial value problem for the
equations of motion of viscous and heat-conductive gases}.  J. Math.
Kyoto Univ. 20 (1980), 67-104.

\bibitem{MN83} A. Matsumura and T. Nishida, \emph{ Initial-boundary value problems for the equations of
motion of compressible viscous and heat-conductive fluids.} Comm.
Math. Phys., 1983, 89, 445-464


\bibitem{MZ10} A. Majumdar and A. Zarnescu, \emph{ Landau-De Gennes
theory of nematic liquid crystals: the Oseen-Frank limit and beyond}.
Arch. Rational Mech. Anal., {196} (2010), 227-280.

%\bibitem{MN}N. J. Mottram, C. Newton, \emph{ Introduction to Q-tensor theory, Technical report,} University of Strathclyde, Department
%of Mathematics (2004).


\bibitem{PZ11} M. Paicu and A. Zarnescu, \emph{ Global existence and regularity for the full coupled Navier-Stokes and Q-tensor system}, SIAM J. Math. Anal., 43 (2011), no. 5, 2009-2049.

\bibitem{PZ12} M. Paicu and A. Zarnescu, \emph{Energy dissipation and regularity for a coupled Navier-Stokes and Q-tensor system}, Arch. Ration. Mech. Anal., 203 (2012), no. 1, 45-67.

\bibitem{TDY03} G. T\'oth, C. Denniston and J. M. Yeomans,
\emph{Hydrodynamics of domain growth in nematic liquid crystals}.
 Phys. Rev. E, 67, 051705 (2003).

%\bibitem{SL} H. Sun, C. Liu, \emph{On energetic variational approaches in modeling the nematic
%liquid crystal flows}. Discrete Contin. Dyn. Syst. 23 (2009), no. 1-2, 455-475.

\bibitem{WY} D. Wang and C. Yu, \emph{Global weak solution and
large-time behavior for the compressible flow of liquid crystals},
Arch. Rational Mech. Anal., 204 (2012), no. 3, 881-915.

\bibitem{XZ11} X. Xu and Z. F. Zhang, \emph{ Global regularity and uniqueness of weak solution for the 2-D liquid crystal flows}. J. Differential Equations 252 (2012), no. 2, 1169-1181.

%\bibitem{LWX11} H. Wu, X. Xu, C. Liu, On the general
%Ericksen--Leslie system: Parodi's relation, well-posedness and
%stability, (2011), preprint, arXiv:1105.2180v5.







\end{thebibliography}
\end{document}